\documentclass[11pt]{article}
\usepackage[pdftex]{graphicx}
\usepackage{bm}
\usepackage{amsmath,amsthm,amssymb,url}
\usepackage{multirow,bigdelim}
\usepackage{amscd}

\usepackage{tikz}
\usetikzlibrary{arrows}

\newtheorem{thm}{Theorem}[section]
\newtheorem{theorem}[thm]{Theorem}
\newtheorem{proposition}[thm]{Proposition}
\newtheorem{lemma}[thm]{Lemma}

\newtheorem{claim}[thm]{Claim}
\newtheorem{corollary}[thm]{Corollary}

\def\hsymb#1{\mbox{\strut\rlap{\smash{\huge$#1$}}\quad}}

\newcommand{\te}{\tilde{e}}
\newcommand{\tu}{\tilde{u}}
\newcommand{\tv}{\tilde{v}}
\newcommand{\bp}{{\bm p}}
\newcommand{\bh}{{\bm h}}
\newcommand{\bmm}{{\bm m}}
\newcommand{\bb}{{\bm b}}

\title{Linking Rigid Bodies Symmetrically}
\author{Bernd Schulze \thanks{Corresponding author; 
Department of Mathematics and Statistics,
University of Lancaster,
Lancaster
LA1 4YF, United Kingdom
(\texttt{b.schulze@lancaster.ac.uk}); phone: +44 (0) 1524 592173.
}
\and
Shin-ichi Tanigawa\thanks{
  Research Institute for Mathematical Sciences, Kyoto University, Kyoto 606-8502 Japan
(\texttt{tanigawa@kurims.kyoto-u.ac.jp}).
  Supported by JSPS Grant-in-Aid for Young Scientist (B), 24740058.}
}

\begin{document}

\renewcommand{\thefootnote}{ }
\renewcommand\thefootnote{\arabic{footnote}}

\maketitle

\begin{abstract}

%It is well known that t
The mathematical theory of rigidity of body-bar and body-hinge frameworks
provides a useful tool for analyzing the rigidity and flexibility of many articulated structures appearing in engineering, robotics and biochemistry. 
In this paper we develop a symmetric extension of this theory which 
%leads to efficient combinatorial algorithms for checking the rigidity properties of
permits a rigidity analysis of body-bar and body-hinge structures with point group symmetries.

The infinitesimal rigidity of body-bar frameworks can naturally be formulated in the language of the exterior (or Grassmann) algebra. Using this algebraic formulation, we derive symmetry-adapted rigidity matrices to analyze the infinitesimal rigidity of body-bar frameworks with Abelian point group symmetries in an arbitrary dimension. In particular, from the patterns of these new matrices, we derive  combinatorial characterizations of infinitesimally rigid body-bar frameworks which are generic with respect to a point group of the form $\mathbb{Z}/2\mathbb{Z}\times \dots \times \mathbb{Z}/2\mathbb{Z}$. 
Our characterizations are given in terms of packings of bases of signed-graphic matroids on quotient graphs.
Finally, we also extend our methods and results to body-hinge frameworks with Abelian point group symmetries in an arbitrary dimension. As special cases of these results, we obtain combinatorial characterizations of infinitesimally rigid body-hinge frameworks with $\mathcal{C}_2$ or  $\mathcal{D}_2$ symmetry - the most
common symmetry groups found in proteins. 
\end{abstract}

%%%%%%%%%%%%%%%%%%%%%%%%%%%%%%%%%%%%%%%%%%%%%%%%%%%%%%%%%%%%%%%%%%%%%%%%%%%%%%%%%%%%%%%%%%%%%%%%%%%%%%%%%%%%%%%%%%%%%%%%%%%%%%%%%%%%%%%%%%

\section{Introduction}
\label{sec:intro}

An important application of rigidity theory is the rigidity and flexibility analysis of biomolecules and proteins, where an ideal molecule is modeled as a {\em body-hinge framework}, 
that is, a structural model consisting of rigid bodies connected, in pairs, by revolute hinges along assigned lines \cite{W1}.
A result by Tay~\cite{tay89,tay91} and Whiteley~\cite{wwmatun,TW1} asserts 
that a generic body-hinge framework is infinitesimally rigid in $\mathbb{R}^3$ if and only if
$5G$ contains six edge-disjoint spanning trees, where $G$ denotes the underlying graph obtained by
identifying each body with a vertex and each hinge with an edge, and $5G$ denotes the graph obtained 
from $G$ by replacing each edge by five parallel copies.  
Based on this result, efficient combinatorial algorithms have been used for analyzing the rigidity properties
of proteins  (see, e.g., \cite{Wbio,leestrei,jrtk}),
%Based on this theorem, an efficient combinatorial algorithm has been used for analyzing the rigidity properties
%of proteins (see, e.g.,?). 
even though body-hinge frameworks arising from molecules do not fit the genericity assumption in Tay-Whiteley's theorem. However, a recent result by Katoh and Tanigawa \cite{kattan} successfully eliminated this assumption, and hence this approach for analyzing the flexibility of proteins is now proven to be mathematically rigorous.

However, many molecules and proteins (as well as many man-made structures such as buildings or mechanical linkages) exhibit non-trivial point group symmetries, and recent work has shown that symmetry can sometimes lead to additional flexibility in a structure (see, e.g., \cite{dimers,portaetal}). Thus, our goal in this paper is to develop a symmetric extension of  generic  rigidity theory which permits a rigidity analysis of structures that possess non-trivial symmetries. 
Our main result is an extension of Tay-Whiteley's theorem which characterizes rigid
symmetric body-hinge structures in terms of a graph packing condition.
This result leads to an efficient combinatorial algorithm for checking the infinitesimal (or static) rigidity properties of 
body-hinge frameworks in the presence of symmetry.  

%The study of how symmetry impacts the rigidity and flexibility of discrete structures has  seen a dramatic
%increase in interest over the last few years,  both within mathematics and applied
%sciences such as robotics or biochemistry (see, e.g.,\cite{owen,portaetal,dimers}).
%\cite{cfgsw,FGsymmax,gsw,jkt,MT1,owen,BS3,schtan,BSWW,dimers,portaetal}. 

The state of the art in the rigidity analysis of symmetric frameworks is as follows (see also \cite{schtan} for a list of recent papers on the subject).
The most basic structure in the context of rigidity theory is  a {\em bar-joint framework}, 
which is composed of rigid bars connected at their ends by flexible joints  \cite{W1}. %Such a framework is called `rigid' if it cannot be deformed continuously into another non-congruent framework while preserving the lengths of all bars, and is called `flexible' otherwise.
%There is also a well known linearized version of rigidity, called infinitesimal rigidity, which coincides with rigidity for generic (i.e., almost all) frameworks \cite{W1}. Infinitesimal rigidity is a stronger notion of rigidity, since every infinitesimally rigid framework is also rigid (the converse, however, is not true in general) \cite{asiroth}.
%  The mathematical model for a $d$-dimensional bar-joint framework is a straight-line realization of a finite simple graph in Euclidean $d$-space, where the vertices and edges represent the joints and bars of the structure, respectively \cite{W1}.
%...
%A number of recent papers have studied when symmetry causes a bar-joint framework to become infinitesimally flexible, or stressed, and when it has no impact 
%(see, e.g., \cite{cfgsw,jkt,MT1,owen,BS3,schtan}, for example). 
%(see \cite{schtan} for a list of recent papers). 
%In particular, 
In \cite{cfgsw} necessary conditions were derived for a symmetric bar-joint framework to be isostatic (i.e., minimally infinitesimally rigid) in $\mathbb{R}^d$ based upon a block-decomposition of the rigidity matrix  (see also \cite{gsw} for the analogous results for body-bar frameworks). Moreover, for some point groups in dimension 2, it was shown in \cite{BS3,BS4} that the conditions in \cite{cfgsw}, together with the standard Laman conditions \cite{Lamanbib}, are also sufficient for a $2$-dimensional bar-joint framework to be isostatic, if it is realized as generic as possible subject to the given symmetry constraints. However, since an infinitesimally rigid symmetric framework typically does not contain an isostatic subframework on the same vertex set with the same symmetry, these results do not provide a general test for infinitesimal rigidity of symmetric frameworks.  

%A major breakthrough in the rigidity analysis of symmetric bar-joint frameworks was recently made by us in \cite{schtan}. 
%In this paper, new tools were established to analyze the infinitesimal rigidity properties of a symmetric bar-joint framework with an Abelian point group by  extending the concept of the `orbit rigidity matrix' introduced in \cite{BSWW} to each of the irreducible representations of the group. 
An advanced approach for the rigidity analysis of symmetric bar-joint frameworks was recently established by us in \cite{schtan}, where we  extended the concept of the `orbit rigidity matrix' introduced in \cite{BSWW} to each of the irreducible representations of the group when the underlying symmetry group is Abelian.  
With the help of these new symmetry-adapted rigidity matrices, combinatorial characterizations of infinitesimally rigid symmetric bar-joint frameworks in the plane were established for several point groups \cite{schtan}. 

A natural and important question is whether one can extend these combinatorial results 
to symmetric frameworks in higher dimensions $d\geq 3$, but for this purpose one  first needs to find a combinatorial characterization of the graphs which form rigid bar-joint frameworks for all generic realizations (without symmetry) in Euclidean $d$-space. 
Unfortunately, finding such a characterization for $d\geq 3$ remains a long-standing open problem in discrete geometry \cite{W1}.

%Unfortunately  the analogous questions for $d\geq 3$ remain long-standing open problems in discrete geometry \cite{W1}.
%Unfortunately, while there exist a number of elegant characterizations of rigid generic bar-joint frameworks in the plane, the analogous questions for $d\geq 3$ remain long-standing open problems in discrete geometry \cite{W1}.

 However, for the special class of \emph{body-bar frameworks} -- which consist of rigid bodies connected by rigid bars (as shown in Figure~\ref{fig:bbpic}(a)) -- there exist neat combinatorial characterizations for generic rigidity in all dimensions \cite{Tay84}. Specifically, it was shown by Tay in 1984 that a generic realization of a multigraph $G$ as a body-bar framework in $d$-space is rigid if and only if $G$ contains ${d+1\choose 2}$ edge-disjoint spanning trees.
As mentioned above, it was independently confirmed by Tay  \cite{tay89,tay91} and Whiteley~\cite{wwmatun} that 
this combinatorial condition also characterizes rigid generic body-hinge frameworks. 
(See also \cite{jj} for further discussions on body-bar-hinge frameworks.)
Also it was recently confirmed that the even more special class of `molecular frameworks' also have the same good combinatorial theory as generic body-bar frameworks \cite{kattan}.
%
%A body-hinge framework  -- which can be described as a special case of a body-bar framework -- consists of rigid bodies that are connected, in pairs, by revolute hinges along assigned lines \cite{W1}. Note that $3$-dimensional body-hinge frameworks are of particular interest in some important practical applications of rigidity.  Moreover, it was recently confirmed that the even more special class of `molecular frameworks' also have the same good combinatorial theory as general body-bar frameworks \cite{kattan}.

In this paper, we present several new results concerning the infinitesimal rigidity of symmetric body-bar and body-hinge frameworks by extending the basic approach for analyzing symmetric bar-joint frameworks described in \cite{schtan} to these structures. 

First, for any Abelian point group $\Gamma$ which acts freely on the bodies of an arbitrary-dimensional body-bar framework, we construct an `orbit rigidity matrix' for each of the irreducible representations of $\Gamma$  using a rigidity formulation of body-bar frameworks in terms of the exterior (or Grassmann) algebra \cite{Tay84,wwmatun,ww} (see Section~\ref{bbblock}). 
%Each of these symmetry-adapted rigidity matrices provides a powerful simplifying tool for detecting symmetry-induced infinitesimal motions which are symmetric  with respect to an irreducible representation of $\Gamma$. In particular, the orbit rigidity matrix  corresponding to the trivial irreducible representation of $\Gamma$ can be used to find hidden `fully-symmetric' infinitesimal motions, which are known to extend to finite (i.e., continuous) symmetry-preserving  motions for body-bar realizations that are generic with respect to $\Gamma$ \cite{BS6}. %(see also \cite{tan}).
 
Note that a body can be considered as a complete bar-joint framework on joints affinely spanning $\mathbb{R}^d$.
In other words, a body-bar framework is a special case of a bar-joint framework
which consists of disjoint complete frameworks connected by  bars. 
Thus the infinitesimal rigidity of body-bar frameworks can be analyzed using rigidity matrices of bar-joint frameworks,
and  one could also use the constructions described in \cite{schtan}  to set up orbit rigidity matrices of symmetric body-bar frameworks. However, the infinitesimal motions of a $d$-dimensional body-bar framework can be expressed in the most natural way using the exterior algebra. In particular, this algebraic formulation allows us to extend the  combinatorial characterizations of rigid generic body-bar frameworks given in \cite{Tay84,wwmatun} to body-bar frameworks which are generic with respect to certain point group symmetries.
 
Specifically, in Section~\ref{bbcombchar}, we  derive combinatorial characterizations of infinitesimally rigid body-bar frameworks which are generic with respect to a point group of the form $\mathbb{Z}/2\mathbb{Z}\times \dots \times \mathbb{Z}/2\mathbb{Z}$. These characterizations are obtained by using signed-graphic matroids and by extending the tree-packing ideas in \cite{ww}.
In Section~\ref{sec:hinge}, we then also extend these results to body-hinge frameworks which are generic with respect to a group $\mathbb{Z}/2\mathbb{Z}\times \dots \times \mathbb{Z}/2\mathbb{Z}$ that acts freely on the structure. Our characterization will be given in terms of a tree-like subgraph packing condition for the quotient graphs, more precisely in terms of bases of signed-graphic matroids on the quotient graphs.

Finally, in Section~\ref{sec:furtherwork}, we discuss some further applications of our results  and methods, and propose some directions for future work.

\section{Body-bar frameworks}
\label{sec:body}

In this section we recall the description of the rigidity matrix of a body-bar framework in terms of the exterior algebra
given by Tay~\cite{Tay84} and Whiteley~\cite{wwmatun}.
To this end we first provide some preliminary facts on Pl{\"u}cker coordinates. 

%As we mentioned, a body-bar framework can be considered as a special case of bar-joint frameworks 
%(by replacing each body with a rigid bar-joint framework),
%and so one could also use the standard rigidity matrix for bar-joint frameworks (see e.g.,).
%However the rigidity matrix in terms of the Grassmann-Cayley algebra will turn out to be essential when deriving 
%combinatorial characterizations in Section~?.
 
%These frameworks are special types of bar-joint frameworks whose generic rigidity (without symmetry)
%is well understood~\cite{kattan,Tay84,TW1,ww,wwmatun}. In particular, there exist complete combinatorial characterizations of generic rigid body-bar frameworks in \emph{all} dimensions.

%As a preliminary, we first begin with linear representations of frame matroids in Subsection~\ref{subsec:dowling}, which will be used in Subsection~?.
%In Subsections~\ref{subsec:body} and \ref{subsec:body_symmetric} we discuss a counterpart of Section~\ref{sec:block}.
%In Subsection~\ref{subsec:body_cyclic} we shall give a combinatorial characterization of the rigidity of body-bar frameworks with
%$\mathbb{Z}/2\mathbb{Z}\times \dots \times \mathbb{Z}/2\mathbb{Z}$ symmetry.

\subsection{Pl{\"u}cker coordinates}
Let $p\in \mathbb{R}^d$.
The homogeneous coordinates of $p$ are denoted by $\hat{p}$, that is,
$\hat{p}=\begin{pmatrix} p \\ 1 \end{pmatrix}\in \mathbb{R}^{d+1}$.
For affinely independent points $p_1,\dots, p_k\in \mathbb{R}^{d}$,
the Pl{\"u}cker coordinates of the (oriented) $k$-simplex determined by $p_1,\dots, p_k$
is the ${d+1\choose k}$-dimensional vector $\hat{p}_1\wedge \dots \wedge \hat{p}_k$
whose entries are the determinants of
the ${d+1\choose k}$ submatrices of size $k\times k$ of the $(d+1)\times k$ matrix
$\begin{pmatrix} \hat{p}_1 & \dots & \hat{p}_k\end{pmatrix}$.
Hence we may index the coordinates of $\hat{p}_1\wedge \dots \wedge \hat{p}_k$ by $k$-tuples
$(i_1,\dots, i_k)$, where $1\leq i_1<\dots <i_k\leq d+1$,
and we may assume that the coordinates are arranged in the lexicographical order of the indices.
The vector $\hat{p}_1\wedge \dots \wedge \hat{p}_k$
is sometimes called a \emph{$k$-extensor} in the context of rigidity theory.

For any $\hat{p}_1,\dots, \hat{p}_k\in \mathbb{R}^{d+1}$,
we may define the wedge product $\hat{p}_1\wedge \dots \wedge \hat{p}_k$
by using the same definition (taking the determinants of the ${d+1\choose k}$ submatrices of size $k\times k$
of $\begin{pmatrix} \hat{p}_1 & \dots & \hat{p}_k\end{pmatrix}$).
Let $Gr(k,d+1)=\{\hat{p}_1 \wedge \dots \wedge \hat{p}_k
\mid \hat{p}_1,\dots,\hat{p}_k\in \mathbb{R}^{d+1}\setminus\{0\}\}$.
Then $Gr(k,d+1)$ linearly spans a ${d+1\choose k}$-dimensional space
which is called the {\em $k$-th exterior power} $\bigwedge^k \mathbb{R}^{d+1}$ of $\mathbb{R}^{d+1}$.

$\bigwedge^k \mathbb{R}^{d+1}$ and $\bigwedge^{d+1-k} \mathbb{R}^{d+1}$ are dual to each other
via the product
$\circ:\bigwedge^k\mathbb{R}^{d+1}\times \bigwedge^{d+1-k}\mathbb{R}^{d+1}\rightarrow \mathbb{R}$
which is defined by
\[
 p\circ q=\sum_{i_1<\dots<i_k}{\rm sign}(\sigma) p_{i_1,\dots, i_k}q_{j_1,\dots,j_{d+1-k}}
\]
for $p\in \bigwedge^k\mathbb{R}^{d+1}$ and $q\in \bigwedge^{d+1-k} \mathbb{R}^{d+1}$,
where $p_{i_1,\dots, i_k}$ and $q_{j_1,\dots,j_{d+1-k}}$ denote
the $(i_1,\dots, i_k)$-th coordinate of $p$ and the $(j_1,\dots,j_{d+1-k})$-th coordinate of $q$, respectively,
$j_1,\dots, j_{d+1-k}$ is the complement of $i_1,\dots, i_k$ in $\{1,2,\dots, d+1\}$ with
$j_1<\dots<j_{d+1-k}$, and
${\rm sign}(\sigma)$ is the sign of the permutation
$\sigma=\begin{pmatrix}i_1 & \dots & i_k & j_1 & \dots & j_{d+1-k} \\ 1 & \dots & k & k+1 & \dots & d+1 \end{pmatrix}$.
For example, for $d=3$ and $k=2$, we have
$p\circ q=p_{1,2}q_{3,4}-p_{1,3}q_{2,4}+p_{1,4}q_{2,3}+p_{2,3}q_{1,4}-p_{2,4}q_{1,3}+p_{3,4}q_{1,2}$.
In general, this product has the following useful property:
A $k$-dimensional linear subspace $X$ and a $(d+1-k)$-dimensional linear subspace $Y$ have a nonzero intersection if and only if the Pl{\"u}cker coordinates $p$  of a $k$-simplex in $X$ and the Pl{\"u}cker coordinates $q$ of
a $(d+1-k)$-simplex in $Y$
satisfy $p\circ q=0$.
This is because
if $p=\hat{p}_1\wedge \dots \wedge \hat{p}_k\in Gr(k,d+1)$ and
$q=\hat{q}_1\wedge \dots \wedge \hat{q}_{d+1-k}\in Gr(d+1-k,d+1)$, then
$p\circ q=\hat{p}_1\wedge \dots \wedge \hat{p}_k\wedge
\hat{q}_1\wedge \dots \wedge \hat{q}_{d+1-k}$,
which is equal to the determinant of a square matrix obtained
by aligning $\hat{p}_1,\dots, \hat{p}_k, \hat{q},\dots, \hat{q}_{d+1-k}$.

Note that both $\bigwedge^{k} \mathbb{R}^{d+1}$ and $\bigwedge^{d+1-k}\mathbb{R}^{d+1}$ are ${d+1\choose k}$-dimensional
linear spaces, and there is a well-known isomorphism between them, known as the Hodge star operator.
Let ${\bf e}_1,\dots, {\bf e}_{d+1}$ be the standard basis of $\mathbb{R}^{d+1}$.
%This product can be seen as an inner product in $\mathbb{R}^{d+1\choose k}$ through the so-called Hodge star-operator.
The Hodge star operator is the linear operator $\ast\colon \bigwedge^k \mathbb{R}^{d+1}\rightarrow \bigwedge^{d+1-k} \mathbb{R}^{d+1}$ defined by
\[
 \ast({\bf e}_{i_1}\wedge \dots \wedge {\bf e}_{i_k})
={\rm sign}(\sigma)  {\bf e}_{j_1}\wedge \dots \wedge {\bf e}_{j_{d+1-k}},
\]
where $j_1,\dots, j_{d+1-k}$ is the complement of $i_1,\dots, i_k$ in $\{1,2,\dots, d+1\}$.
%In our situation, the Hodge star-operator just returns a $(d+1-k)$-extensor of the orthogonal complement.
%In particular,  $\langle \alpha, \beta \rangle=0$ if and only if $\alpha \cdot \ast \beta=0$.
For example, if $d=3$ and $k=2$, $\ast q=(q_{3,4},-q_{2,4},q_{2,3}, q_{1,4},-q_{1,3},q_{1,2})$ for $q=(q_{1,2},q_{1,3},q_{1,4}, q_{2,3},q_{2,4},q_{3,4})$.

By identifying $\bigwedge^k \mathbb{R}^{d+1}$ with $\bigwedge^{d+1-k} \mathbb{R}^{d+1}$ via $\ast$ and identifying $\bigwedge^k \mathbb{R}^{d+1}$ with  $\mathbb{R}^{d+1\choose k}$,
we can regard $\circ$ as an inner product $\langle\cdot,\cdot\rangle$ in $\mathbb{R}^{d+1\choose k}$
since $p\circ q=\langle p, \ast q\rangle$.
%which is the standard inner product of $\bp$ and $\ast \bq$ for $\bp\in \bigwedge^k \mathbb{R}^{d+1}$
%and $\bq\in \bigwedge^{d+1-k}\mathbb{R}^{d+1}$.
%Thus, in the subsequent discussion, we shall denote $p\circ q$ simply by $\langle p, q \rangle$.
%If $p$ and $q$ are the Pl{\"u}cker coordinates of a $k$-simplex $X$ and a $(d+1-k)$-simplex $Y$, then
%$X\cap Y\neq \emptyset$ if and only if $\langle p,q\rangle=0$.
%where  $\langle p,q\rangle=0$  if and only if
%$X\cap Y=\{0\}$, for a $k$-dimensional linear subspace $X$ and a $(d+1-k)$-dimensional linear subspace $Y$
%with the Pl{\"u}cker coordinates $[\bp]$ and $[\bq]$.

\subsection{Rigidity matrices of body-bar frameworks}
\label{subsec:body}
A body-bar framework is a structural model consisting of rigid bodies which are pairwise connected by rigid bars as shown in Figure~\ref{fig:bbpic}.
We identify each body with a vertex and each bar with an edge
to indicate the underlying incidence of bodies and bars in the body-bar framework (see also Figure~\ref{fig:bbpic}(b)). More formally,
we define a {\em body-bar framework} to be a pair $(G,\bb)$ of an undirected multigraph $G$ and a bar-configuration\footnote{Note that
an edge $\{u,v\}$ is an unordered pair, whereas  $\hat{p}_{e,u}\wedge \hat{p}_{e,v}$ is ordered (i.e., $\hat{p}_{e,u}\wedge \hat{p}_{e,v}=-\hat{p}_{e,v}\wedge \hat{p}_{e,u}$).
Formally, we should define $\bb$ in such a way that $\bb: E(G)\rightarrow Gr(2,d+1)/\{1,-1\}$,
but for the sake of simplicity of the description we will use the definition of (\ref{eq:bar_conf}).
In fact, for deciding whether the framework is infinitesimally rigid or not, we just need the linear space spanned
by $\hat{p}_{e,u}\wedge \hat{p}_{e,v}$ for each bar.}
\begin{align}
\label{eq:bar_conf}
\begin{split}
\bb:\quad E(G) \quad &\rightarrow Gr(2,d+1) \\
e=\{u,v\} &\mapsto \hat{p}_{e,u}\wedge \hat{p}_{e,v}.
\end{split}
\end{align}
That is, $\bb(\{u,v\})=\hat{p}_{e,u}\wedge \hat{p}_{e,v}$ indicates the Pl{\"u}cker coordinates of
the bar connecting the point $p_{e,u}$ in the body $u$ and the point $p_{e,v}$ in the body $v$.
(See again Figure~\ref{fig:bbpic} for an example.)

\begin{figure}[t]
\begin{center}
\begin{tikzpicture}[very thick,scale=0.9]
\tikzstyle{every node}=[circle, draw=black, fill=white, inner sep=0pt, minimum width=5pt];
\filldraw[fill=black!20!white, draw=black, thin](0,-1.8)ellipse(1.2cm and 0.35cm);
\filldraw[fill=black!20!white, draw=black, thin](0,-0.2)ellipse(1.2cm and 0.35cm);

\node [circle, draw=black!20!white, shade, ball color=black!40!white, inner sep=0pt, minimum width=7pt](p1) at (-0.3,-1.95) {};
\node [circle, draw=black!20!white, shade, ball color=black!40!white, inner sep=0pt, minimum width=7pt](p2) at (-0.76,-1.85) {};
%\node [circle, draw=black!20!white, shade, ball color=black!40!white, inner sep=0pt, minimum width=7pt](p3) at (-0.3,-1.63) {};
\node [circle, draw=black!20!white, shade, ball color=black!40!white, inner sep=0pt, minimum width=7pt](p4) at (0.16,-1.83) {};
\node [circle, draw=black!20!white, shade, ball color=black!40!white, inner sep=0pt, minimum width=7pt](p5) at (0.79,-1.74) {};
\node [circle, draw=black!20!white, shade, ball color=black!40!white, inner sep=0pt, minimum width=7pt](p6) at (0.43,-1.95) {};

\node [circle, draw=black!20!white, shade, ball color=black!40!white, inner sep=0pt, minimum width=7pt](t1) at (0.18,-0.025) {};
%\node [circle, draw=black!20!white, shade, ball color=black!40!white, inner sep=0pt, minimum width=7pt](t2) at (-0.78,-0.25) {};
\node [circle, draw=black!20!white, shade, ball color=black!40!white, inner sep=0pt, minimum width=7pt](t3) at (-0.5,-0.1) {};
\node [circle, draw=black!20!white, shade, ball color=black!40!white, inner sep=0pt, minimum width=7pt](t4) at (0.43,-0.23) {};
\node [circle, draw=black!20!white, shade, ball color=black!40!white, inner sep=0pt, minimum width=7pt](t5) at (0.74,-0.11) {};
\node [circle, draw=black!20!white, shade, ball color=black!40!white, inner sep=0pt, minimum width=7pt](t6) at (-0.21,-0.3) {};

\draw(p1)--(t1);
\draw(p2)--(t3);
%\draw(p3)--(t2);
\draw(p4)--(t4);
\draw(p5)--(t5);
\draw(p6)--(t6);
\node [draw=white, fill=white] (a) at (0,-3) {(a)};

\node [draw=white, fill=white,rectangle] (a) at (-0.9,-1) {$e$};
\node [draw=white, fill=white,rectangle] (b) at (-0.7,0.36) {$p_{e,u}$};
\node [draw=white, fill=white,rectangle] (c) at (-0.9,-2.3) {$p_{e,v}$};
\node [draw=white, fill=white,rectangle] (c) at (0,0.5) {$u$};
\node [draw=white, fill=white,rectangle] (c) at (0,-2.45) {$v$};
\node [draw=white, fill=white,rectangle] (c) at (-2,-0.9) {$(G,\bb)$};

\end{tikzpicture}
\hspace{1.5cm}
     \begin{tikzpicture}[very thick,scale=1]
\tikzstyle{every node}=[circle, draw=black, fill=white, inner sep=0pt, minimum width=5pt];
       \path (0.9,-1.7) node (p4) [label =  below: $v$] {} ;
       \path (0.9,-0.3) node (p8) [label =  above: $u$] {} ;
       \draw (p4) -- (p8);
       \path
(p4) edge [bend right=22] (p8);
  \path
(p4) edge [bend right=44] (p8);
 \path
(p4) edge [bend left=44] (p8);
  \path
(p4) edge [bend left=22] (p8);

\node [draw=white, fill=white,rectangle] (c) at (0.2,-1) {$G$};

\node [draw=white, fill=white] (a) at (0.9,-3) {(b)};
       \end{tikzpicture}

\end{center}
\vspace{-0.3cm}
\caption{A (non-symmetric) $3$-dimensional body-bar framework $(G,\bb)$ (a) and its underlying multigraph $G$ (b). We may think of each of the two bodies of $(G,\bb)$ as a complete bar-joint framework on the end-points of the bars attached to the body.}
\label{fig:bbpic}
\end{figure}

An infinitesimal motion of a body-bar framework $(G,\bb)$ is defined as
$\bmm:V(G)\rightarrow \mathbb{R}^{d+1\choose 2}$
satisfying
\begin{equation}
\label{eq:inf_body}
\langle \bmm(u)-\bmm(v), \bb(e)\rangle=0 \qquad \text{for all } e=\{u,v\}\in E(G).
\end{equation}
It should be noted that (\ref{eq:inf_body}) is essentially equivalent to 
the first-order length constraint appearing in the infinitesimal (or static) analysis of bar-joint frameworks,
as $\bb(e)$ denotes (the coordinates of) the direction from $\bp_{e,u}$ to $\bp_{e,v}$.

Observe that $\bmm$ is an infinitesimal motion if $\bmm(u)=\bmm(v)$ for all $u,v\in V(G)$.
Such a motion is called a trivial (infinitesimal) motion.
Thus, the set of trivial motions forms a ${d+1\choose 2}$-dimensional linear space.
$(G,\bb)$ is called {\em infinitesimally rigid} if all infinitesimal motions of $(G,\bb)$ are trivial.

The {\em rigidity matrix} $R(G,\bb)$ of $(G,\bb)$ is the $|E(G)|\times {d+1\choose 2}|V(G)|$ matrix
defined by
\begin{displaymath} \bordermatrix{& & & & u & & & & v & & & \cr & & & &  & & \vdots & &  & & &
\cr e=\{u,v\} & 0 & \ldots &  0 & \bb(e) & 0 & \ldots & 0 & -\bb(e) &  0 &  \ldots&  0 \cr & & & &  & & \vdots & &  & & &
}
\textrm{,}\end{displaymath}
 that is, $R(G,\bb)$ is the matrix associated with the linear system (\ref{eq:inf_body}).
Note that $(G,\bb)$ is infinitesimally rigid if and only if ${\rm rank}\ R(G,\bb)={d+1\choose 2}(|V(G)|-1)$.

A bar-configuration $\bb$ is said to be \emph{generic} if
$\{p_{e,v}, p_{e,u}\mid e=\{u,v\}\in E(G)\}$ is algebraically independent over $\mathbb{Q}$.
Tay~\cite{Tay84} proved that if $\bb$ is generic, then
$(G,\bb)$ is infinitesimally rigid if and only if 
$G$ contains ${d+1\choose 2}$ edge-disjoint spanning trees.
%
%$|E(G)|={d+1\choose 2}|V(G)|-{d+1\choose 2}$ and
%$|F|\leq {d+1\choose 2}|V(F)|-{d+1\choose 2}$ for any nonempty $F\subseteq E(G)$,
%equivalently $G$ can be decomposed into 
We shall give a symmetric extension of this result in Theorem~\ref{thm:comb_body}. 
%%%%%%%%%%%%%%%%%%%%%%%%%%%%%%%%%%%%%%%%%%%%%%%%%%%%%%%%%%%%%%

\section{Symmetric body-bar frameworks}
\subsection{Symmetric multigraphs}
\label{subsec:multigraphs}

In order to develop a rigidity theory for symmetric body-bar frameworks, we first need to introduce some basic concepts concerning symmetric graphs.

Let $G$ be a finite simple graph.
An {\em automorphism} of $G$ is a permutation $\pi:V(G)\rightarrow V(G)$ such that
$\{u,v\}\in E(G)$ if and only if $\{\pi(u),\pi(j)\}\in E(G)$.
The set of all automorphisms of $G$ forms a subgroup of the symmetric group on $V(G)$,
known as the {\em automorphism group} ${\rm Aut}(G)$ of $G$.
An {\em action} of a group $\Gamma$ on $G$ is a group homomorphism $\theta:\Gamma \rightarrow {\rm Aut}(G)$.
An action $\theta$ is called {\em free} on $V(G)$ (resp., $E(G)$)
if  $\theta(\gamma)(v)\neq v$ for any $v\in V(G)$ (resp., $\theta(\gamma)(e)\neq e$ for any $e\in E(G)$) and
any non-identity $\gamma\in \Gamma$.
We say that a graph $G$ is {\em $\Gamma$-symmetric} (with respect to $\theta$)
if $\Gamma$ acts on $G$ by $\theta$.
Throughout the paper, we only consider the case when $\theta$ is free on $V(G)$, and we
omit to specify the action $\theta$, if it is clear from the context.
We then denote $\theta(\gamma)(v)$ by  $\gamma v$.

For a $\Gamma$-symmetric graph $G$,
the {\em quotient graph} $G/\Gamma$ is a multigraph whose vertex set is the set $V(G)/\Gamma$
of vertex orbits and whose edge set is the set $E(G)/\Gamma$ of edge orbits.
Several distinct graphs may have the same quotient graph.
However, if we assume that the underlying action is free on $V(G)$,
then a gain labeling  makes the relation one-to-one as explained below.

Let $H$ be a directed graph which may contain multiple edges and loops, and let $\Gamma$ be a group.
A {\em $\Gamma$-gain graph} (or $\Gamma$-labeled graph)  is a pair $(H,\psi)$ in which each edge is associated with an
element of $\Gamma$ via a {\em gain function} $\psi:E(H)\rightarrow \Gamma$.

Given a $\Gamma$-symmetric graph $G$, 
we arbitrarily choose a vertex $v$ as a representative vertex from each vertex orbit.
Then each orbit is of the form $\Gamma v=\{gv\mid g\in \Gamma\}$.
If the action is free,
an edge orbit connecting $\Gamma u$ and $\Gamma v$ in $G/\Gamma$
can be written as $\{\{gu,ghv\}\mid g\in \Gamma \}$
for a unique $h\in \Gamma$.
We then orient the edge orbit from $\Gamma u$ to $\Gamma v$ in $G/\Gamma$
and assign to it the gain $h$.
In this way, we obtain {\em the quotient $\Gamma$-gain graph}, denoted by $(G/\Gamma,\psi)$.
$(G/\Gamma,\psi)$ is unique up to choices of representative vertices.
Figure \ref{fig:ggex} illustrates an example, where $\Gamma$ is the reflection group $\mathcal{C}_s$.

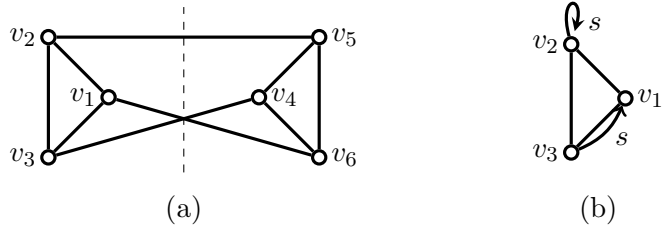
\begin{figure}[htp]
\begin{center}
  \begin{tikzpicture}[very thick,scale=1]
\tikzstyle{every node}=[circle, draw=black, fill=white, inner sep=0pt, minimum width=5pt];
     \path (-1,0) node (p1)  [label = left: $v_1$]{} ;
    \path (1,0) node (p2) [label = right: $v_4$]{}  ;
    \path (-1.8,0.8) node (p3) [label = left: $v_2$]{} ;
     \path (1.8,0.8) node (p4)  [label = right: $v_5$]{} ;
     \path (-1.8,-0.8) node (p5)[label = left: $v_3$]{} ;
     \path (1.8,-0.8) node (p6) [label = right: $v_6$]{};
\draw (p1)  --  (p3);
        \draw (p1)  --  (p5);
              \draw (p6)  --  (p1);

        \draw (p2)  --  (p4);
        \draw (p2)  --  (p6);
        \draw (p4)  --  (p3);
        \draw (p5)  --  (p2);

        \draw (p5)  --  (p3);
        \draw (p4)  --  (p6);
        \draw[dashed,thin] (0,-1)  --  (0,1.2);

           \node [rectangle, draw=white, fill=white] (b) at (0,-1.5) {(a)};
        \end{tikzpicture}
        \hspace{2cm}
         \begin{tikzpicture}[very thick,scale=0.9]
\tikzstyle{every node}=[circle, draw=black, fill=white, inner sep=0pt, minimum width=5pt];
     \path (-1.2,0.15) node (p1) [label = right: $v_1$]{} ;
       \path (-2,0.95) node (p3) [label = left: $v_2$]{} ;
      \path (-2,-0.65) node (p5) [label = left: $v_3$]{} ;
      \draw (p1)  --  (p3);
        \draw (p1)  --  (p5);
        \draw (p3)  --  (p5);
        \path
(p3) edge [loop above,->, >=stealth,shorten >=2pt,looseness=26] (p3);
\path
(p5) edge [->,bend right=22] (p1);
\node [rectangle, draw=white, fill=white] (b) at (-1.25,-0.45) {$s$};
\node [rectangle, draw=white, fill=white] (b) at (-1.65,1.25) {$s$};
\node [rectangle, draw=white, fill=white] (b) at (-1.6,-1.5) {(b)};
        \end{tikzpicture}
        \end{center}
         \caption{A $\mathcal{C}_s$-symmetric graph (a) and its quotient gain graph (b), where $\mathcal{C}_s=\{id,s\}$. For simplicity, we omit the direction and the label of every edge with gain $id$.}
 \label{fig:ggex}
\end{figure}

The map $c:G\rightarrow H$ defined by $c(gv)=v$ and $c(\{gu,g\psi(e)v\})=(u,v)$ is called a {\em  covering map}.
In order to avoid confusion, throughout the paper, a vertex or an edge in a quotient gain graph $H$
is denoted with the mark tilde, e.g., $\tilde{v}$ or $\tilde{e}$.
Then the fiber $c^{-1}(\tilde{v})$ of a vertex $\tilde{v}\in V(H)$ and the fiber $c^{-1}(\tilde{e})$
of an edge $\tilde{e}\in E(H)$
coincide with a vertex orbit and an edge orbit, respectively, in $G$.

Since the underlying graph of body-bar frameworks are multigraphs,  
we need to extend the definition of symmetric graphs from simple graphs to multigraphs.
This can be done in a straightforward fashion: a multigraph $G$ is 
{\em $\Gamma$-symmetric with respect to $\theta:\Gamma\rightarrow {\rm Aut}(G)$}
if $\theta$ is a group homomorphism.
By fixing a representative vertex for each vertex orbit
we can define the {\em quotient $\Gamma$-gain graph} in the analogous way.
However, in the case of multigraphs,
distinct $\Gamma$-symmetric multigraphs may lead to the same $\Gamma$-gain graph.

To see this, consider a $\mathbb{Z}/2\mathbb{Z}$-gain graph with one vertex $\tilde{v}$ and
one loop $\te$.
Let the fiber of $\tilde{v}$ be $\{v,v'\}$. Then, if $\Gamma$  acts freely on the edge set,
the fiber of $\te$ is the set consisting of two parallel edges joining $v$ and $v'$; otherwise,
if $\Gamma$ does not act freely on the edge set,
the fiber of $\te$ is the set consisting of the single edge $\{v,v'\}$ (see Figure~\ref{fig:bbpicfree}).

%can be the set consisting of the single edge $\{v,v'\}$ or the set consisting of two parallel edges joining $v$ and $v'$, depending on whether $\Gamma$  acts freely on the edge set (see Figure~\ref{fig:bbpicfree}).

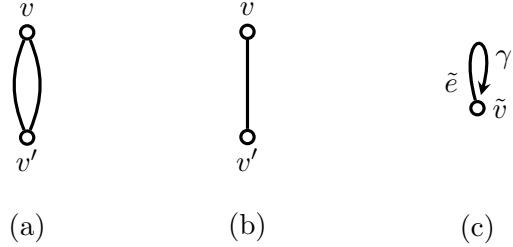
\begin{figure}[htp]
\begin{center}
\begin{tikzpicture}[very thick,scale=1]
\tikzstyle{every node}=[circle, draw=black, fill=white, inner sep=0pt, minimum width=5pt];
       \path (0,-1.7) node (p1)  {} ;
 \path (0,-0.3) node (p2)  {} ;

\path
(p1) edge [bend right=22] (p2);

\path
(p1) edge [bend left=22] (p2);

\node [draw=white, fill=white] (a) at (0,-2.9) {(a)};

\node [draw=white, fill=white,rectangle] (c) at (0,0) {$v$};
\node [draw=white, fill=white,rectangle] (c) at (0,-2) {$v'$};

\end{tikzpicture}
\hspace{2cm}
\begin{tikzpicture}[very thick,scale=1]
\tikzstyle{every node}=[circle, draw=black, fill=white, inner sep=0pt, minimum width=5pt];
       \path (0,-1.7) node (p1)  {} ;
 \path (0,-0.3) node (p2)  {} ;

\draw (p1) -- (p2);

\node [draw=white, fill=white] (a) at (0,-2.9) {(b)};

\node [draw=white, fill=white,rectangle] (c) at (0,0) {$v$};
\node [draw=white, fill=white,rectangle] (c) at (0,-2) {$v'$};\end{tikzpicture}
\hspace{2cm}
   \begin{tikzpicture}[very thick,scale=1]
\tikzstyle{every node}=[circle, draw=black, fill=white, inner sep=0pt, minimum width=5pt];
       \path (0,-1.3) node (p1)  {} ;
       %\path (0.9,-0.3) node (p8) [label =  above: $u$] {} ;

 \path
(p1) edge [loop above,->, >=stealth,shorten >=2pt,looseness=46] (p1);

\node [draw=white, fill=white,rectangle] (c) at (-0.35,-0.95) {$\tilde{e}$};
\node [draw=white, fill=white,rectangle] (c) at (0.3,-1.3) {$\tilde{v}$};

\node [draw=white, fill=white,rectangle] (c) at (0.35,-0.65) {$\gamma$};

\node [draw=white, fill=white] (a) at (0,-2.9) {(c)};
       \end{tikzpicture}

\end{center}
\vspace{-0.3cm}
\caption{Two distinct $\Gamma$-symmetric multigraphs ((a),(b)) which may have the same quotient $\Gamma$-gain graph (c). In the case of (a) we have $L=\emptyset$, whereas in the case of (b) we have $L=\{\te\}$.}
\label{fig:bbpicfree}
\end{figure}

Therefore, to impose a one-to-one correspondence between $\Gamma$-symmetric multigraphs
and quotient graphs (up to the choice of representative vertices),
we equip the quotient graph $H$ with
a gain labeling $\psi:E(H)\rightarrow \Gamma$ and also with the set $L$ of loops in $H$ that correspond to edge orbits of $G$ on which $\Gamma$ does not act freely via $\theta$.
Note that $L\subseteq \{\te\in E(H)\mid \text{ $\te$ is a loop with $\psi(\te)^2=id$} \}$.
 (See also Fig.~\ref{fig:bbpic2} for another example.)

%%%%%%%%%%%%%%%%%%%%%%%%%%%%%%%%%%%%%%%%%%%%%%%%%%%%%%%%%%%%%%%%%%%%%%%%%%%%%%%%%%%%%%%%%%%%%%%%%%%%%%%%%%%%%%%%%%%%

\subsection{Symmetric body-bar frameworks}
\label{subsec:body_symmetric}
%%For simplicity of description, we shall assume thoughout these two subsections that
%%$\theta:\Gamma\rightarrow {\rm Aut}(G)$ is also a free action  on the edge set
%%(and then $L=\emptyset$).
%%The case when $\theta$ is not free on the edge set will be discussed in Subsection~.
%
%As we noted earlier, a body-bar framework is a special case of a bar-joint framework, and hence
%we can define $\Gamma$-symmetric body-bar frameworks using the definition of
%$\Gamma$-symmetric bar-joint frameworks.
%In the following, we will give an interpretation in terms of the notation given in Section~\ref{subsec:body}.

Let us first recall some basic facts regarding group actions on  exterior product spaces.
Suppose that $\Gamma$ has an orthogonal representation $\hat{\tau}:\Gamma\rightarrow O(\mathbb{R}^{d+1})$.
Then there is a unique representation
$\hat{\tau}^{(2)}:\Gamma\rightarrow O(\bigwedge^2 \mathbb{R}^{d+1})$ induced by $\hat{\tau}$ such that
$\hat{\tau}^{(2)}(\hat{p}\wedge \hat{q})=\hat{\tau}(\hat{p})\wedge \hat{\tau}(\hat{q})$
for any $\hat{p}\wedge \hat{q}\in Gr(2,d+1)$.

In the following, we will give an explicit definition of $\hat{\tau}^{(2)}$.
For an orthogonal matrix $A$ of size $(d+1)\times (d+1)$,
we define a matrix $A^{(2)}$  of size ${d+1\choose 2}\times {d+1 \choose 2}$ as follows.
Assume that each row and each column of $A^{(2)}$ is indexed
by pairs $(i,j)$ and $(k,l)$, where $1\leq i<j\leq d+1$ and $1\leq k<l\leq d+1$,
respectively.
Then the entries of $A^{(2)}$ are given by
\[
 A^{(2)}[(i,j),(k,l)]={\rm det}\ A_{i,j}^{k,l},
\]
where $A_{i,j}^{k,l}$ is the $2\times 2$-submatrix
of $A$ induced by the $i$-th and the $j$-th rows and by the $k$-th and the $l$-th columns.

For $\hat{\tau}$, define $\hat{\tau}^{(2)}$ by $\hat{\tau}^{(2)}(\gamma)=(\hat{\tau}(\gamma))^{(2)}$ for every
$\gamma\in \Gamma$.
Then it is known that $\hat{\tau}^{(2)}$ is a well-defined representation of $\Gamma$.
Moreover, if $\hat{\tau}$ is an orthogonal representation, then
$\hat{\tau}^{(2)}$ is also an orthogonal representation.
(To see this, consider a matrix $A=\hat{\tau}(\gamma)$.
Then  we have
$(A^{(2)})^{\top}=(A^{\top})^{(2)}$ by definition, and $(A^{-1})^{(2)}=(A^{(2)})^{-1}$ since
$\hat{\tau}^{(2)}$ is a group representation.
Therefore, we have
$(A^{(2)})^{\top}=(A^{\top})^{(2)}=(A^{-1})^{(2)}=(A^{(2)})^{-1}$.)

For any $1\leq k\leq d+1$, one can define an orthogonal representation
$\hat{\tau}^{(k)}:\Gamma\rightarrow O(\bigwedge^k\mathbb{R}^{d+1})$ in
the same manner.
%We then remark that, for any $p\in \bigwedge^k \mathbb{R}^{d+1}$,
%$q\in \bigwedge^{d+1-k}\mathbb{R}^{d+1}$ and $\gamma \in \Gamma$,
%we have
%$$\langle \hat{\tau}^{(k)}(\gamma)p, q\rangle={\rm det}\
%\hat{\tau}(\gamma) \langle p,
%(\hat{\tau}^{(d+1-k)}(\gamma))^{-1} q\rangle.$$
%Indeed, if $p=\hat{p}_1\wedge \dots \hat{p}_k\in Gr(k,d+1)$ and $q=\hat{q}_1\wedge
%\dots \wedge \hat{q}_{d+1-k}\in Gr(d+1-k,d+1)$
%and if we denote $\hat{\tau}(\gamma)$ by $A$, then
%we have $\langle
%A^{(k)} p, A^{(d+1-k)}q\rangle
%=A \hat{p}_1\wedge \dots \wedge A p_k\wedge
%A q_1\wedge \dots \wedge A q_{d+1-k}=
%{\rm det} A (\hat{p}_1\wedge \dots \wedge p_k\wedge
%q_1\wedge \dots \wedge q_{d+1-k})=
%{\rm det} A \langle p,q\rangle$,
%which means
%$\langle A^{(k)} p, q\rangle
%=\langle A^{(k)}(\gamma)p,
%A^{(d+1-k)} (A^{(d+1-k)})^{-1}q\rangle
%={\rm det} A \langle p, (A^{(d+1-k)})^{-1}q\rangle$.

Now let us return to symmetric body-bar frameworks.
Let $\Gamma$ be a finite group and let $\tau:\Gamma\rightarrow O(\mathbb{R}^{d})$.
We define the \emph{augmented representation} $\hat{\tau}:\Gamma\rightarrow O(\mathbb{R}^{d+1})$
 by $\hat{\tau}(\gamma)=\begin{pmatrix} \tau(\gamma) & 0 \\ 0 & 1 \end{pmatrix}$.
We say that a body-bar framework $(G,\bb)$ is {\em $\Gamma$-symmetric} (with respect to $\theta$ and $\tau$)
if $G$ is $\Gamma$-symmetric with an action $\theta:\Gamma\rightarrow {\rm Aut}(G)$ and
\begin{equation}
\label{eq:body_bar_sym0}
\hat{\tau}(\gamma)\hat{p}_{e,v}=\hat{p}_{\theta(\gamma)e,\theta(\gamma)v} \qquad
\text{for all $\gamma\in \Gamma$ and $e=\{u,v\}\in E(G)$. }
\end{equation}
This implies
\begin{equation}
\label{eq:body_bar_sym}
\bb(\theta(\gamma)e)=\hat{\tau}^{(2)}(\gamma)\bb(e) \qquad \text{for all } e=\{u,v\}\in E(G).
\end{equation}

We denote by $P_V:\Gamma\rightarrow GL(\mathbb{R}^V)$
the linear representation of $\Gamma$ induced by $\theta$ over $V(G)$,
that is,
$P_V(\gamma)$ is the
permutation matrix of the permutation $\theta(\gamma)$ of $V(G)$.
Specifically, $P_V(\gamma)=[\delta_{i,\theta(\gamma)(j)})]_{i,j}$, where $\delta$ denotes the Kronecker delta symbol.
Similarly, let $P_E:\Gamma\rightarrow GL(\mathbb{R}^E)$
be the linear representation of $\Gamma$ consisting of
permutation matrices of permutations induced by $\theta$ over $E(G)$.

The following is the counterpart of \cite[Theorem~3.1]{schtan}, where $\otimes$ stands for the Kronecker product (the tensor product).
We omit the identical proof.
\begin{theorem}
\label{thm:body_block}
Let $\Gamma$ be a finite group with $\tau:\Gamma\rightarrow O(\mathbb{R}^d)$,
$G$ be a $\Gamma$-symmetric graph with a free action $\theta$ on $V(G)$ and
$(G,\bb)$ be a $\Gamma$-symmetric body-bar framework with respect to $\theta$ and $\tau$.
Then $R(G,\bb)$ is an intertwiner of $\hat{\tau}^{(2)}\otimes P_V$ and $P_E$,
i.e., $R(G,\bb) (\hat{\tau}^{(2)}\otimes P_V) = P_E R(G,\bb)$.
\end{theorem}
%\begin{proof}
%The proof is identical to that of Theorem~\ref{thm:block} if we simply replace $\bp(e)$ with $\bb(e)$.
%\end{proof}

It follows from Theorem~\ref{thm:body_block} and Schur's lemma that there are non-singular matrices $S$ and $T$ such that $T^{\top} R(G,\bb) S$ is block-diagonalized as
\begin{equation}
\label{rigblocks}
T^{\top}R(G,\bb)S:=\widetilde{R}(G,\bb)
=\left(\begin{array}{ccc}\widetilde{R}_{0}(G,\bb) & & \mathbf{0}\\ & \ddots & \\\mathbf{0} &  &
\widetilde{R}_{r}(G,\bb) \end{array}\right)\textrm{,}
\end{equation}
where the submatrix block $\widetilde{R}_i(G,\bb)$ corresponds to the irreducible representation $\rho_i$ of $\Gamma$.

%%%%%%%%%%%%%%%%%%%%%%%%%%%%%%%%%%%%%%%%%%%%%%%%%%%%%%%%%%%%%%%%%%%%%%%%%%%%%%%%%%%%%%%%%%%%%%%%%%%%%%%%%%%%%%%%%%

\subsection{Block-diagonalization of the rigidity matrix for body-bar frameworks with Abelian symmetry}
\label{bbblock}
In this subsection we shall derive explicit entries of each block in the block-diagonalized form of the rigidity matrix. 
The corresponding work for bar-joint frameworks was done in our previous paper~\cite{schtan},
and here we just confirm that the same technique can be applied.

Throughout the subsequent discussion,  
$\Gamma$ is assumed to be an Abelian group of the form
$\mathbb{Z}/k_1\mathbb{Z}\times \dots \times \mathbb{Z}/k_l\mathbb{Z}$
for some positive integers $k_1,\dots, k_l$.
Thus, we may denote each element of $\Gamma$ by ${\bm i}=(i_1,\dots,i_l)$, where
 $0\leq i_1\leq k_1,\dots, 0\leq i_l\leq k_l$, and regard $\Gamma$ as an additive group.

Let $k=|\Gamma|=k_1 k_2 \dots k_l$.
It is an elementary fact from group representation theory that
$\Gamma$ has $k$ non-equivalent irreducible representations
which are denoted by  $\{\rho_{\bm j}\colon {\bm j}\in \Gamma\}$.
Specifically, for each ${\bm j}\in \Gamma$, $\rho_{\bm j}$ is defined by
\begin{align} \label{eq:abelian_rho}
\rho_{{\bm j}}:\Gamma&\rightarrow \mathbb{C}/\{0\}  \nonumber\\
{\bm i}&\mapsto \omega_1^{i_1j_1}\cdot\omega_2^{i_2j_2} \cdot \ldots \cdot\omega_l^{i_lj_l} ,
\end{align}
where $\omega_t=e^{\frac{2\pi \sqrt{-1}}{k_t}}$, $t=1,\ldots, l$.
To cope with such representations, we extend the underlying field from $\mathbb{R}$ to $\mathbb{C}$ if required.
%where each element of $\Gamma$ is denoted by a tuple ${\bm j}=(j_1,\dots, j_l)$,
%$\rho_{\bm j}$ is an irreducible representation of $\Gamma$ (as given in (\ref{eq:abelian_rho})),
%and $(H,\psi)$ denotes the quotient $\Gamma$-gain graph of $G$ with covering map $c:G\rightarrow H$.

Let $(G,\bb)$ be a $\Gamma$-symmetric body-bar framework and $(H,\psi)$ be the quotient 
$\Gamma$-gain graph of $G$ with covering map $c:G\rightarrow H$.
Let $K$ be the set of all maps $\bmm$ of the form $\bmm:V(G)\rightarrow \mathbb{R}^{d+1\choose 2}$.
Then the rigidity matrix $R(G,\bb)$ represents a linear map from $K$ to a linear space of dimension $|E(G)|$. 
Also $\hat{\tau}^{(2)}\otimes P_V$ acts on $K$.
An infinitesimal motion $\bmm\in K$ is said to be
{\em $\rho_{\bm j}$-symmetric}  if
\begin{equation}
\label{eq:body_mo_sym}
\bmm({\bm i} v)=\hat{\tau}^{(2)}_{\bm j}({\bm i})\bmm(v) \qquad
\text{ for each ${\bm i} \in \Gamma$ and $v\in V(G)$}
\end{equation}
where $\hat{\tau}^{(2)}_{\bm j}$ denotes the matrix representation of $\Gamma$ defined by
\begin{equation}
\hat{\tau}^{(2)}_{\bm j}:\ {\bm i} \ \mapsto
\ \rho_{\bm j}({\bm i})^{-1} \cdot \hat{\tau}^{(2)}({\bm i}).
\end{equation}
Let $K_{\bm j}=\{\bmm\in K\mid \bmm \text{ satisfies (\ref{eq:body_mo_sym})} \}$.
The following lemma validates the definition of $\rho_{\bm j}$-symmetric infinitesimal motions. 
\begin{lemma}
\label{lem:invariance}
$K_{\bm j}$ is the $\rho_{\bm j}$-invariant subspace of $K$ under the action $\hat{\tau}^{(2)}\otimes P_V$.
In other words, 
\begin{equation}
(\hat{\tau}^{(2)}\otimes P_V)({\bm i})\cdot \bmm= \rho_{\bm j}({\bm i})\cdot \bmm
\end{equation}
for every $\bmm\in K_{\bm j}$ and every ${\bm i}\in \Gamma$.
\end{lemma}
\begin{proof}
For each $v\in V(G)$, 
$((\hat{\tau}^{(2)}\otimes P_V)({\bm i})\cdot \bmm)(v)=\hat{\tau}^{(2)}({\bm i})\cdot \bmm({\bm i}^{-1}\cdot v)=\hat{\tau}^{(2)}({\bm i})\cdot \hat{\tau}_{\bm j}^{(2)}({\bm i})^{-1}\cdot \bmm(v)=\rho_{\bm j}({\bm i})\cdot \bmm(v)$. 
\end{proof}

Now let us consider how to compute the dimension of the set of $\rho_{\bm j}$-symmetric infinitesimal motions.
Recall that for a body-bar framework $(G,\bb)$,
a map $\bmm\in K$ is an infinitesimal motion if and only if
\begin{equation}
\label{eq:body_inf}
 \langle \bb(e), \bmm(u)-\bmm(v)\rangle=0 \qquad \text{for all } e=\{u,v\}\in E(G).
\end{equation}
This system of linear equations for $\bmm$ is redundant
if $\bmm$ is restricted to be $\rho_{\bm j}$-symmetric.
Since  the edge orbit associated with $\te\in E(H)$ is
 $c^{-1}(\te)=\{\{\gamma u,\gamma \psi_{\te} v\}|\,\gamma\in \Gamma\}$,
 (\ref{eq:body_inf}) can be written as
\begin{equation}
\label{eq:body_inf2_1}
\langle \bb(\{\gamma u, \gamma \psi_{\te} v\}), \bmm(\gamma u)-\bmm(\gamma \psi_{\te} v)\rangle=0
\qquad (\gamma\in \Gamma)
\end{equation}
for each edge orbit.
By (\ref{eq:body_bar_sym}) and (\ref{eq:body_mo_sym}),
(\ref{eq:body_inf2_1}) becomes
\begin{equation}
\label{eq:body_inf2}
\langle \hat{\tau}^{(2)}(\gamma)\bb(\{u,\psi_{\te} v\}),
\hat{\tau}^{(2)}_{\bm j}(\gamma)
(\bmm(u)-\bmm(\psi_{\te} v))\rangle=0 \qquad (\gamma\in \Gamma).
\end{equation}
%Since $\hat{\tau}^{(2)}$ is an orthogonal representation,
These $k$ equations are equivalent to the single equation
\begin{equation}
\label{eq:body_j_inf}
%\langle \bb(\{u,\psi_{\te} v\}), \bmm(u)-\bmm(\psi_{\te}v)\rangle=
\langle \bb(\{u,\psi_{\te} v\}), \bmm(u)-\hat{\tau}^{(2)}_{\bm j}(\psi_{\te}) \bmm(v)\rangle=0
\end{equation}
for each edge orbit.

This implies that the analysis can be done on the quotient $\Gamma$-gain graph $(H,\psi)$.
To see this, let us define
the motion $\tilde{\bmm}(\tv)$ of a vertex $\tv\in V(H)$
to be the motion $\bmm(v)$ of the representative vertex (body) $v$ of the vertex orbit $c^{-1}(\tv)$.
Also, for a bar-configuration $\bb$ of the form (\ref{eq:bar_conf}),
let $\tilde{\bb}:E(H)\rightarrow Gr(2,d+1)$ be given by
\begin{equation}
\label{eq:bar_quotient}
\tilde{\bb}(\tilde{e})=\bb(\{u,\psi_{\te} v\})=\hat{p}_{e,u}\wedge \hat{p}_{e,\psi_{\te} v}
\qquad (\tilde{e}\in E(H))
\end{equation}
for each $\tilde{e}\in E(H)$,
where $e, u, v$ denote the representative edge and vertices in the corresponding orbits $c^{-1}(\tilde{e}), c^{-1}(\tilde{v}), c^{-1}(\tilde{u})$.

Then, for a $\Gamma$-gain graph $(H,\psi)$ and $\tilde{\bb}:E(H)\rightarrow Gr(2,d+1)$,
a map $\tilde{\bmm}:V(H)\rightarrow \mathbb{R}^d$ is said to be a
$\rho_{\bm j}$-symmetric motion of $(H,\psi,\tilde{\bb})$ if
\begin{equation}
\label{eq:body_j_inf_quot}
\langle \tilde{\bb}(\te),
\tilde{\bmm}(\tu)-\hat{\tau}^{(2)}_{\bm j}(\psi_{\te}) \tilde{\bmm}(\tv)\rangle=0 \qquad \text{for all } \te=(\tu,\tv)\in E(H).
\end{equation}
We define the {\em $\rho_{\bm j}$-orbit rigidity matrix},
denoted by $O_{\bm j}(H,\psi,\tilde{\bb})$,  to be the  matrix of size $|E(H)|\times {d+1\choose 2}|V(H)|$
associated with the system (\ref{eq:body_j_inf_quot}),
in which each vertex has the corresponding ${d+1\choose 2}$ columns,
and the row corresponding to
$\te=(\tu,\tv)\in E(H)$ has the form
\[\begin{array}{ccccc}
       & \overbrace{\hspace{10mm}}^{\tu} & & \overbrace{\hspace{25mm}}^{\tv} & \\
      0\dots0 & \tilde{\bb}(\te) & 0\dots0 & -(\hat{\tau}^{(2)}_{\bm j}(\psi_{\te}))^{-1}\tilde{\bb}(\te) & 0\dots0 \\
    \end{array}\]
where each vector is assumed to be transposed. If $\te$ is a loop at $\tv$, then the entries of $\tv$
become the sum of the two entries:
\[\begin{array}{ccccc}
        & \overbrace{\hspace{40mm}}^{\tv} & \\
      0\dots0 & (I_{d+1\choose 2}- (\hat{\tau}^{(2)}_{\bm j}(\psi_{\te}))^{-1}) \tilde{\bb}(\te) & 0\dots0  \\
    \end{array}\]
The following proposition asserts that one can reduce the problem of computing the rank of each block in the block-diagonalization to the computation of the rank of $O_{\bm j}(H,\psi,\tilde{\bb})$.
\begin{proposition}
Let $\Gamma$ be an Abelian group, $(G,\bb)$ be a $\Gamma$-symmetric framework in $\mathbb{R}^d$,
and $(H,\psi)$ be the quotient $\Gamma$-gain graph of $G$.
Then, for each ${\bm j}\in \Gamma$
\[
 {\rm rank}\ \widetilde{R}_{\bm j}(G,\bb)={\rm rank}\ O_{\bm j}(H,\psi,\tilde{\bb}).
\]
\end{proposition}
\begin{proof}
The detailed description for the corresponding proposition for bar-joint frameworks is given in \cite[Section 4]{schtan}. Hence we only give a sketch of the proof.

One can easily check that $K=\bigoplus_{{\bm j}\in \Gamma} K_{\bm j}$. Hence, by Lemma~\ref{lem:invariance}, the kernel of each block $\tilde{R}_{\bm j}(G,\bb)$ is equal to the set of $\rho_{\bm j}$-symmetric infinitesimal motions.
From the above discussion this set has a one-to-one correspondence with the kernel of 
$O_{\bm j} (H,\psi,\tilde{\bb})$.
\end{proof}

%%%%%%%%%%%%%%%%%%%%%%%%%%%%%%%%%%%%%%%%%%%%%%%%%%%%%%%%%%%%%%%%%%%%%%%%%%%%%%%%%%%%%%%%%%%%%%%%%%

\subsection{$\Gamma$-generic Frameworks}
For a discrete point group $\Gamma\subseteq O(\mathbb{R}^d)$, 
let $\mathbb{Q}_{\Gamma}$ be the field generated by $\mathbb{Q}$ and by the entries of the matrices in $\Gamma$.

In this subsection we shall give a formal definition of {\em generic} bar-configurations under symmetry.
To this end it should be noted that for $\tilde{\bm b}(\te)$, defined in (\ref{eq:bar_quotient}),
there exists a geometric relation between $\hat{p}_{e,v}$ and $\hat{p}_{e,\psi_{\te}v}$
if $\Gamma$ does not  act freely on the edge orbit corresponding to $\te$.
To see this,
let us consider a $\Gamma$-symmetric  body-bar framework $(G,\bb)$
for which the underlying action $\theta$ is not free on $E(G)$.
Recall that for a quotient gain graph $(H,\psi)$,
$L$ denotes the set of loops in $E(H)$ corresponding
to the edge orbits of $G$ on which $\Gamma$ does not act freely (cf.~Section~\ref{subsec:multigraphs}).
If we denote the representative edge of $\te\in L$ by $e=\{v, \psi_{\te}v\}\in E(G)$, then
we have $\theta(\psi_{\te})e=e$.
Together with (\ref{eq:body_bar_sym0}), this  implies that
$$\hat{\tau}(\psi_{\te})\hat{p}_{e,v}=\hat{p}_{\theta(\psi_{\te})e, \theta(\psi_{\te})v}
=\hat{p}_{e,\theta(\psi_{\te})v}.$$
Thus, for each $\te\in L$, $\tilde{\bb}(\te)$ is of the form
\begin{equation}
\label{eq:not_free_loop}
\tilde{\bb}(\te)=\hat{p}\wedge \hat{\tau}(\psi_{\te})\hat{p}
\end{equation}
for some $\hat{p}\in \mathbb{R}^{d+1}\setminus \{0\}$.
(In contrast, for $\te\in E(H)\setminus L$, $\tilde{\bb}(\te)$ has the form
$\tilde{\bb}(\te)=\hat{p} \wedge \hat{q}$,
where $\hat{p}$ and $\hat{q}$ are \emph{any} two points in $\mathbb{R}^{d+1}\setminus \{0\}$.)

Thus, for a discrete point group $\Gamma$,
a $\Gamma$-symmetric body-bar framework $(G,\bb)$ is said to be {\em $\Gamma$-generic}
if there is a set
$\{\hat{p}_{\te}, \hat{q}_{\te}\mid \te\in E(H)\setminus L\}\cup \{\hat{p}_{\te}\mid \te\in L\}$
of points in $\mathbb{R}^{d+1}$
such that the set of coordinates is algebraically independent over $\mathbb{Q}_{\Gamma}$ and
$\tilde{\bb}$ is of the form
\[
 \tilde{\bb}(\te)=\begin{cases}
\hat{p}_{\te}\wedge \hat{q}_{\te} & \text{if $\te\in E(H)\setminus L$} \\
\hat{p}_{\te}\wedge \hat{\tau}(\psi_{\te})\hat{p}_{\te} & \text{if $\te\in L$}
\end{cases}\qquad (\te\in E(H)).
\]

A loop $\tilde{e}$ is called a {\em zero loop} in $O_{\bm j}(H,\psi,\tilde{\bb})$
if the row corresponding to $\te$ is a zero vector in $O_{\bm j}(H,\psi,\tilde{\bb})$.
Due to the above geometric restriction, a loop $\te$ of $L$ may be a zero loop even if $(G,\bb)$ is $\Gamma$-generic.
\begin{proposition}
\label{prop:zero_loop2}
Let $\Gamma$ be an Abelian group, $\tau:\Gamma\rightarrow O(\mathbb{R}^d)$ be a faithful representation, and
$(G,\bb)$ be a $\Gamma$-symmetric body-bar framework.
Then a loop $\te$ in $L$ is a zero loop in $O_{\bm j}(H,\psi,\tilde{\bb})$ if and only if
$\rho_{\bm j}(\psi_{\te})=-1$.
%A loop $\te$ that is not contained in $L$ is a zero loop if and only if
%$\bb(\te)\in {\rm ker}(I_{d+1\choose 2}-\hat{\tau}^{(2)}_{\bm j}(\psi_{\te}))$.
\end{proposition}
\begin{proof} Since  $\te\in L$, $\tilde{\bb}(\te)$ is of the form
$\tilde{\bb}(\te)=\hat{p}\wedge \hat{\tau}(\psi_{\te})\hat{p}$ for some non-zero $\hat{p}\in \mathbb{R}^{d+1}$, by (\ref{eq:not_free_loop}). We have
\begin{equation}
\label{eqn:zl}
\big(I_{d+1\choose 2}- \hat{\tau}^{(2)}_{\bm j}(\psi_{\te}^{-1})\big) \tilde{\bb}(\te)= \hat{p}\wedge \big(\hat{\tau}(\psi_{\te})+ \rho_{\bm j}(\psi_{\te}^{-1}) \hat{\tau}(\psi_{\te}^{-1}) \big)\hat{p}
\end{equation}
by $
\big(I_{d+1\choose 2}- \hat{\tau}^{(2)}_{\bm j}(\psi_{\te}^{-1})\big) \tilde{\bb}(\te)
 = \big(I_{d+1\choose 2}- \rho_{\bm j}(\psi_{\te}^{-1})(\hat{\tau}^{(2)}(\psi_{\te}^{-1})\big) \hat{p}\wedge \hat{\tau}(\psi_{\te})\hat{p}
 = \hat{p}\wedge \hat{\tau}(\psi_{\te})\hat{p} - \rho_{\bm j}(\psi_{\te}^{-1}) \hat{\tau}(\psi_{\te}^{-1})\hat{p}\wedge \hat{p} 
 =  \hat{p}\wedge \big(\hat{\tau}(\psi_{\te})+ \rho_{\bm j}(\psi_{\te}^{-1}) \hat{\tau}(\psi_{\te}^{-1}) \big)\hat{p}.
$
Thus, since $\psi_{\te}^2=id$,
if $\rho_{\bm j}(\psi_{\te})=-1$, then we have $\rho_{\bm j}(\psi_{\te}^{-1})=-1$ and $\tau(\psi_{\te})=\tau(\psi_{\te}^{-1})$, implying that $\te$ is a zero loop, by (\ref{eqn:zl}).

Conversely, let $\rho_{\bm j}(\psi_{\te}^{-1})=\omega$ and $\tau(\psi_{\te})=A$. We show that if $\te$ is a zero loop, then $\omega=-1$. Note that
$$\hat{\tau}(\psi_{\te})+\rho_{\bm j}(\psi_{\te}^{-1})\hat{\tau}(\psi_{\te}^{-1})= \left(\begin{array}{c c } A & 0\\ 0 & 1 \end{array}\right)+\omega \left(\begin{array}{c c } A^{-1} & 0\\ 0 & 1 \end{array}\right)= \left(\begin{array}{c c } A+\omega A^{-1} & 0\\ 0 & 1+\omega\end{array}\right).$$
If $\te$ is a zero loop, (\ref{eqn:zl}) implies that
$$\left(\begin{array}{c c } A+\omega A^{-1} & 0\\ 0 & 1+\omega\end{array}\right)= cI_{d+1} \textrm{ for some } c\in\mathbb{C}.$$
%In the first case, we have $\omega=-1$ and $A-A^{-1}=0$, and hence $A^2=I_d$, as desired. In the second case,
We then have $c=1+\omega$ and $A+\omega A^{-1}=(1+\omega)I_d$. The latter equation implies $\omega=1$ or $\omega=-1$ (see the proof of Proposition 4.3 in \cite{schtan}). Since $\omega\neq 1$ (otherwise $G$ contains a loop), we obtain $\omega=-1$.
\end{proof}

%To be added (as given in the attached note).

%%%%%%%%%%%%%%%%%%%%%%%%%%%%%%%%%%%%%%%%%%%%%%%%%%%%%%%%%%%%%%%%%%%%%%%%%%%%%%

\subsection{Example}\label{sec:exambb}

Consider the $3$-dimensional body-bar framework $(G,\bb)$  depicted in Figure~\ref{fig:bbpic2} (a) which consists of two bodies connected by six bars. Such a structure is also known as a `Stewart platform' in the engineering community. The framework in Figure~\ref{fig:bbpic2} (a) is $\mathcal{C}_s$-symmetric (with respect to $\theta$ and $\tau$), where $\mathcal{C}_s=\{id, s\}$, and the corresponding quotient gain graph $(H,\psi)$ is shown in Figure~\ref{fig:bbpic2} (b). Recall that $\mathcal{C}_s$ has only two non-equivalent irreducible representations $\rho_0$ and $\rho_1$.
%, as shown in Table~\ref{tabcs}. 
Let us construct the $\rho_1$-symmetric (or `anti-symmetric') orbit rigidity matrix $O_{1}(H,\psi,\tilde{\bb})$ of $(G,\bb)$. This matrix describes the `anti-symmetric' infinitesimal rigidity properties of $(G,\bb)$, where, by (\ref{eq:body_mo_sym}), an infinitesimal motion $\bmm$ of $(G,\bb)$ is anti-symmetric if
 \begin{equation}
\bmm(\theta(s)(v))=\hat{\tau}^{(2)}_{1}(s)\bmm(v) \qquad
\text{ for all $v\in V(G)$.}\nonumber
\end{equation}
Suppose that the reflection plane of $s$ is the $x-y$-plane, that is,  $\hat{\tau}(s)= \left(\begin{array}{c c c c} 1 & 0 & 0 & 0\\ 0 & 1 & 0 & 0\\ 0& 0 & -1 & 0\\ 0& 0 & 0&1\end{array}\right)$. Then (using the lexicographical order for the row and column indices of $\hat{\tau}^{(2)}(s)$) we have
$$\hat{\tau}^{(2)}_{1}(s)=\rho_1(s)\cdot \hat{\tau}^{(2)}(s)=(-1)\cdot \left(\begin{array}{c c c c c c} 1 & 0 & 0 & 0 & 0 & 0\\ 0 & -1 & 0 & 0 & 0& 0\\ 0& 0 & 1 & 0& 0 & 0\\ 0& 0 & 0& -1 & 0 & 0 \\ 0& 0 & 0 & 0& 1 & 0\\ 0& 0 & 0& 0 & 0 & -1 \end{array}\right).$$

\vspace{0.2cm}
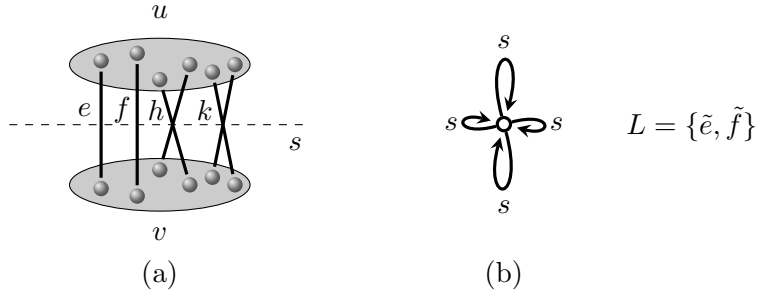
\begin{figure}[htp]
\begin{center}
\begin{tikzpicture}[very thick,scale=1]
\tikzstyle{every node}=[circle, draw=black, fill=white, inner sep=0pt, minimum width=5pt];
\filldraw[fill=black!20!white, draw=black, thin](0,-1.8)ellipse(1.2cm and 0.35cm);
\filldraw[fill=black!20!white, draw=black, thin](0,-0.2)ellipse(1.2cm and 0.35cm);

\node [circle, draw=black!20!white, shade, ball color=black!40!white, inner sep=0pt, minimum width=7pt](p1) at (-0.3,-1.95) {};
\node [circle, draw=black!20!white, shade, ball color=black!40!white, inner sep=0pt, minimum width=7pt](p2) at (-0.78,-1.85) {};
\node [circle, draw=black!20!white, shade, ball color=black!40!white, inner sep=0pt, minimum width=7pt](p3) at (0,-1.6) {};
\node [circle, draw=black!20!white, shade, ball color=black!40!white, inner sep=0pt, minimum width=7pt](p4) at (0.4,-1.8) {};
\node [circle, draw=black!20!white, shade, ball color=black!40!white, inner sep=0pt, minimum width=7pt](p5) at (0.7,-1.7) {};
\node [circle, draw=black!20!white, shade, ball color=black!40!white, inner sep=0pt, minimum width=7pt](p6) at (1,-1.8) {};

\node [circle, draw=black!20!white, shade, ball color=black!40!white, inner sep=0pt, minimum width=7pt](t1) at (-0.3,-0.05) {};
\node [circle, draw=black!20!white, shade, ball color=black!40!white, inner sep=0pt, minimum width=7pt](t2) at (-0.78,-0.15) {};
\node [circle, draw=black!20!white, shade, ball color=black!40!white, inner sep=0pt, minimum width=7pt](t3) at (0,-0.4) {};
\node [circle, draw=black!20!white, shade, ball color=black!40!white, inner sep=0pt, minimum width=7pt](t4) at (0.4,-0.2) {};
\node [circle, draw=black!20!white, shade, ball color=black!40!white, inner sep=0pt, minimum width=7pt](t5) at (0.7,-0.3) {};
\node [circle, draw=black!20!white, shade, ball color=black!40!white, inner sep=0pt, minimum width=7pt](t6) at (1,-0.2) {};

\node [draw=white, fill=white,rectangle] (a) at (-0.05,-0.8) {$h$};
\node [draw=white, fill=white,rectangle] (a) at (0.6,-0.8) {$k$};
\draw(p1)--(t1);
\draw(p2)--(t2);
\draw(p3)--(t4);
\draw(p4)--(t3);
\draw(p5)--(t6);
\draw(p6)--(t5);
\node [draw=white, fill=white] (a) at (0,-3) {(a)};

\node [draw=white, fill=white,rectangle] (a) at (-1,-0.8) {$e$};
\node [draw=white, fill=white,rectangle] (a) at (-0.5,-0.8) {$f$};

%\node [draw=white, fill=white,rectangle] (b) at (-0.7,0.36) {$p_{e,u}$};
%\node [draw=white, fill=white,rectangle] (c) at (-0.9,-2.3) {$p_{e,v}$};
\node [draw=white, fill=white,rectangle] (c) at (0,0.5) {$u$};
\node [draw=white, fill=white,rectangle] (c) at (0,-2.45) {$v$};
%\node [draw=white, fill=white,rectangle] (c) at (-2,-0.9) {$(G,\bb)$};
\draw[dashed, thin](-2,-1)--(2,-1);
\node [draw=white, fill=white,rectangle] (c) at (1.8,-1.25) {$s$};

\end{tikzpicture}
\hspace{1.5cm}
   \begin{tikzpicture}[very thick,scale=1]
\tikzstyle{every node}=[circle, draw=black, fill=white, inner sep=0pt, minimum width=5pt];
       \path (0,-1) node (p1)  {} ;
       %\path (0.9,-0.3) node (p8) [label =  above: $u$] {} ;

 \path
(p1) edge [loop above,->, >=stealth,shorten >=2pt,looseness=46] (p1);
 \path
(p1) edge [loop left,->, >=stealth,shorten >=2pt,looseness=26] (p1);
 \path
(p1) edge [loop below,->, >=stealth,shorten >=2pt,looseness=46] (p1);
 \path
(p1) edge [loop right,->, >=stealth,shorten >=2pt,looseness=26] (p1);

%\node [draw=white, fill=white,rectangle] (c) at (-0.35,-0.65) {$\tilde{e}$};
%\node [draw=white, fill=white,rectangle] (c) at (-0.3,-1.5) {$\tilde{f}$};

\node [draw=white, fill=white,rectangle] (c) at (-0.7,-1) {$s$};
\node [draw=white, fill=white,rectangle] (c) at (0.7,-1) {$s$};
\node [draw=white, fill=white,rectangle] (c) at (0,0.1) {$s$};
\node [draw=white, fill=white,rectangle] (c) at (0,-2.1) {$s$};

\node [draw=white, fill=white,rectangle] (c) at (2.5,-1) {$L=\{\tilde{e},\tilde{f}\}$};

\node [draw=white, fill=white] (a) at (0,-3) {(b)};
       \end{tikzpicture}

\end{center}
\vspace{-0.3cm}
\caption{A  body-bar framework in 3D (also known as a `Stewart platform') with reflection symmetry $\mathcal{C}_s$  (a) and its quotient gain graph (b). }
\label{fig:bbpic2}
\end{figure}

The anti-symmetric orbit rigidity matrix $O_{1}(H,\psi,\tilde{\bb})$ is the following $4\times 6$ matrix:
\begin{displaymath}\bordermatrix{
                &\tilde{u}\cr
                (\tilde{h};s)&  \big(I_6-\hat{\tau}^{(2)}_1(s)^{-1}\big)\tilde{\bb}(\tilde{h})  \cr
                (\tilde{k};s)&  \big(I_6-\hat{\tau}^{(2)}_1(s)^{-1}\big)\tilde{\bb}(\tilde{k})  \cr
 (\te;s)& 0 \ 0 \ 0 \ 0  \ 0 \ 0 \cr
 (\tilde{f};s)& 0 \ 0  \ 0 \ 0 \ 0 \ 0 \cr
}
                \end{displaymath}
where an edge $\tilde{a}$ with label $\gamma$ is denoted by $(\tilde{a};\gamma)$, and $\tilde{\bb}(\tilde{a})=\hat{p}_{a,u}\wedge \hat{p}_{a,\theta(s)(u)}=\hat{p}_{a,u}\wedge \hat{p}_{a,v}$. Note that by Proposition~\ref{prop:zero_loop2}, the loops $\te$ and $\tilde{f}$ in $L$ are zero loops in $O_{1}(H,\psi,\tilde{\bb})$, and hence $O_{1}(H,\psi,\tilde{\bb})$ has only two non-trivial rows.

While generic realizations  of the multigraph $G$ as a body-bar framework (without symmetry) are clearly rigid (in fact, isostatic), as six `independent' bars remove the six relative degrees of freedom between the two bodies, we will show in the next section that $\mathcal{C}_s$-generic realizations  of $G$ as a body-bar framework such as the one in Figure~\ref{fig:bbpic2} (a) are  infinitesimally flexible with an anti-symmetric infinitesimal flex.

%BS: Maybe include the $3$-dimensional $6$-ring (i.e., the 3D body-bar framework consisting of 6 bodies alternating with 6 "hinges" (sets of 5 bars))? If realized with reflection symmetry, where the mirror plane goes through two of the hinges (and none of the bodies), fixing exactly two bars, we get the counts

 %$\rho_0$ orbit matrix: $e_0=2\cdot 5+2\cdot 3= 16 > 15 = 6\cdot b_0 -3 $. Hence there exists a fully symmetric self stress.

% $\rho_1$ orbit matrix:  $e_0=2\cdot 5+2\cdot 2= 14 < 15 = 6\cdot b_0 -3$. Hence there exists an anti-symmetric motion.

% So, both Tay's non-symmetric counts and the forced-symmetric counts for rigidity are satisfied, and we need the anti-symmetric counts to detect flexibility.

%%%%%%%%%%%%%%%%%%%%%%%%%%%%%%%%%%%%%%%%%%%%%%%%%%%%%%%%%%%%%%%%%%%%%%%%%%%%%%%%%%%%%%%%%%%%%%%%%%%%%%%%%%

\section{Combinatorial characterizations for body-bar frameworks}
\label{bbcombchar}
For a $\Gamma$-symmetric body-bar framework $(G,\bb)$ with respect to $\theta$ and $\tau$,
we say that $(G,\bb)$ is {\em $\Gamma$-regular} if $R(G,\bb)$ has maximal rank
among all $\Gamma$-symmetric body-bar realizations of $G$. Note that a $\Gamma$-generic framework is clearly $\Gamma$-regular.
In this subsection we give a combinatorial characterization of infinitesimally rigid
$\Gamma$-regular  body-bar frameworks for
$\Gamma$ isomorphic to $\mathbb{Z}/2\mathbb{Z}\times \dots \times \mathbb{Z}/2\mathbb{Z}$.
For this we use a result from matroid theory which we explain in Section~\ref{subsec:dowling}. We then give the combinatorial characterization in Section~\ref{subsec:body_chara}.

\subsection{Signed-graphic matroids}
\label{subsec:dowling}
Let $(H,\psi)$ be a $\mathbb{Z}/2\mathbb{Z}$-gain graph, where we treat $\mathbb{Z}/2\mathbb{Z}$ as a multiplicative group $\mathbb{Z}/2\mathbb{Z}=\{-1,1\}$. 
Then a cycle in $H$ is called {\em positive} (resp.~{\em negative}) if the number of  edges with negative gains is even (resp.~odd). In the signed-graphic matroid ${\cal G}(H,\psi)$, an edge set $F\subseteq E(H)$ is independent if and only if each connected component contains at most one cycle, which is negative if it exists.
The signed-graphic matroid is a special case of {\em frame matroids} (or, bias matroid) on gain graphs, see, e.g.,\cite{oxley} for more details. 

It is known that ${\cal G}(H,\psi)$ is representable over $\mathbb{R}$ as follows.
To each $\te=(\tilde{i},\tilde{j})\in E(H)$, we associate a vector  $x_{\te}\in \mathbb{F}^{V(H)}$
defined by
\begin{equation*}
\label{eq:dowling_rep}
x_{\te}(\tv)=\begin{cases}
-\psi(\te) & \text{ if } \tv=\tilde{i} \\
1 & \text{ if } \tv=\tilde{j} \\
0 & \text{ otherwise}
\end{cases}
\end{equation*}
if $\te$ is not a loop, and
\begin{equation*}
\label{eq:dowling_rep_loop}
x_{\te}(\tv)=\begin{cases}
1-\psi(\te) & \text{ if } \tv=\tilde{i} \\
0 & \text{ otherwise}
\end{cases}
\end{equation*}
if $\te$ is a loop attached at $\tilde{i}$.
Then we consider a $|E(H)|\times |V(H)|$ matrix $I(H,\psi)$
consisting of rows $x_{\te}$ for all $\te\in E(H)$.
The matrix is identical to the incidence matrix of $H$, except that the entry becomes $1$ instead of $-1$
if the corresponding edge has label $-1$.
It is known that
$F\subseteq E(H)$ is independent in ${\cal G}(H,\psi)$  if and only if
the set of row vectors of  $I(H,\psi)$ associated with $F$ is linearly independent
(see, e.g., \cite{oxley}).

%%%%%%%%%%%%%%%%%%%%%%%%%%%%%%%%%%%%%%%%%%%%%%%%%%%%%%%%%%%%%%%%%%%%%%%%%%%%

\subsection{Combinatorial characterizations}
\label{subsec:body_chara}

Suppose that $\Gamma=\mathbb{Z}/2\mathbb{Z}\times \dots \times \mathbb{Z}/2\mathbb{Z}$.
Suppose also that $\Gamma$ acts on $\mathbb{R}^d$ via $\tau:\Gamma\rightarrow O(\mathbb{R}^d)$.
We may assume that $\tau(\gamma)$ is a diagonal matrix with entries in $\{-1,0,1\}$ for each $\gamma\in \Gamma$.
Then
$\hat{\tau}^{(2)}_{\bm g}(\gamma)$ is a diagonal matrix of size ${{d+1\choose 2}\times {d+1\choose 2}}$
in which each diagonal entry is either $1$ or $-1$ for each ${\bm g}\in \Gamma$.
(Note that for the sake of clarity, we deviate from our previous notation here and  use ${\bm g}$ instead of ${\bm j}$.)
Therefore, $\hat{\tau}^{(2)}_{\bm g}$ can be decomposed into ${d+1 \choose 2}$ one-dimensional representations
as follows:
\[
 \hat{\tau}^{(2)}_{\bm g}=\bigoplus_{1\leq i< j\leq d+1} \tau_{\bm g}^{i,j},
\] where
\[
\tau_{\bm g}^{i,j}:\Gamma\rightarrow \mathbb{Z}/2\mathbb{Z}=\{-1,1\}
\]
(where $\mathbb{Z}/2\mathbb{Z}$ is regarded as a multiplicative group).
Then each $\tau_{\bm g}^{i,j}$ induces
a labeling function
\begin{align*}
 \psi_{\bm g}^{i,j}:E(H)\ &\rightarrow \ \mathbb{Z}/2\mathbb{Z}=\{-1,1\} \\
\te\ &\mapsto \ \tau_{\bm g}^{i,j}(\psi(\te)).
\end{align*}
The resulting labeling functions $\psi_{\bm g}^{i,j}\ (1\leq i<j\leq d+1)$ over the quotient graph $H$ are called
{\em the labeling functions induced by $\hat{\tau}^{(2)}_{\bm g}$}.

\begin{theorem}
\label{thm:comb_body}
Let $\Gamma=\mathbb{Z}/2\mathbb{Z}\times \dots \times \mathbb{Z}/2\mathbb{Z}$,
$(G,\bb)$ be a $\Gamma$-generic body-bar framework with respect to a faithful
$\tau:\Gamma\rightarrow O(\mathbb{R}^d)$ and a free $\theta:\Gamma\rightarrow {\rm Aut}(G)$ on $V(G)$, and
$(H,\psi)$ be the corresponding quotient $\Gamma$-gain graph. Further, let
${\bm g}\in \Gamma$ and
$(H_{\bm g},\psi)$ be the $\Gamma$-gain graph obtained from $(H,\psi)$
by removing all loops $\te\in L$ with $\rho_{\bm g}(\psi_{\te})=-1$.
The linear matroid determined by the row vectors in $O_{\bm g}(H,\psi,\tilde{\bb})$
is the matroid union of ${\cal G}(H_{\bm g},\psi_{\bm g}^{i,j})$ over all $1\leq i< j\leq d+1$,
where $\psi_{\bm g}^{i,j}$ are the labeling functions induced by $\hat{\tau}^{(2)}_{\bm g}$,
followed by adjoining all the removed loops of $H$ as loops (in the matroidal sense).

In other words, the following are equivalent:
\begin{description}
\item[(i)] ${\rm rank}\ O_{\bm g}(H,\psi,\tilde{\bb})=|E(H_{\bm g})|$;
\item[(ii)] For any nonempty $F\subseteq E(H_{\bm g})$,
\[
 |F|\leq {d+1\choose 2}|V(F)|-{d+1\choose 2}+\sum_{1\leq i<j\leq d+1}\alpha_{\bm g}^{i,j}(F),
\]
where
\begin{equation}
\label{eq:alpha}
 \alpha_{\bm g}^{i,j}(F)=\begin{cases}
1 & \text{ if $F$ contains a negative cycle in $(H_{\bm g},\psi_{\bm g}^{i,j})$} \\
0 & \text{ otherwise}
\end{cases}
\end{equation}
\item[(iii)] $H_{\bm g}$ can be decomposed into ${d+1\choose 2}$ subgraphs $H_{1,2},\dots, H_{d,d+1}$ such that
for every $1\leq i<j\leq d+1$,
every connected component of $(H_{i,j},\psi_{\bm g}^{i,j})$ contains no cycle or just one cycle, which is negative (with respect to the labeling $\psi_{\bm g}^{i,j}$).
\end{description}
\end{theorem}
\begin{proof}
We first remark that \textbf{(ii)} and \textbf{(iii)} are equivalent by Nash-Williams' matroid union theorem.
To see this, recall that in the frame matroid ${\cal G}(H_{\bm g},\psi_{\bm g}^{i,j})$,
an edge set $F$ is independent if and only if each connected component of $F$ contains no cycle or just one
cycle, and the cycle is negative if it exists.
Therefore, condition \textbf{(iii)} is nothing but the necessary and sufficient condition for $E(H_{\bm g})$ to be
independent in the union $\bigvee_{1\leq i<j\leq d+1} {\cal G}(H_{\bm g},\psi_{\bm g}^{i,j})$.

Further, it follows from the independence condition of  ${\cal G}(H_{\bm g},\psi_{\bm g}^{i,j})$
that the rank function $r_{\bm g}^{i,j}:E(H_{\bm g})\rightarrow \mathbb{Z}$ of
${\cal G}(H_{\bm g},\psi_{\bm g}^{i,j})$ can be written as
\[
 r_{\bm g}^{i,j}(F)=\sum_{X:\text{ component of } F}(|V(X)|-1+\alpha_{\bm g}^{i,j}(X)) \qquad (F\subseteq E(H_{\bm g})),
\]
where the sum is taken over all connected components $X$ of $F$.
By the matroid union theorem\footnote{We use Nash-Williams' theorem as follows.
Suppose that ${\cal M}_1,\dots, {\cal M}_k$
are matroids on the same ground set $S$ with rank functions $r_1,\dots, r_k$, respectively.
Then Nash-Williams' matroid union theorem says that
the rank function $r:S\rightarrow \mathbb{Z}$ of the union $\bigvee_{1\leq i\leq k} {\cal M}_i$
can be written as $r(X)=\min_{X'\subseteq X}\{|X'|+\sum_{1\leq i\leq k}r_i(X\setminus X')\}$.
Note that $S$ is independent in the union if and only if $|X|\leq r(X)$ for every $X\subseteq S$,
but the latter condition is equivalent to
$|X|\leq \sum_{1\leq i\leq k}r_i(X)$ for every $X\subseteq S$.},
$E(H_{\bm g})$ is independent in $\bigvee_{1\leq i<j\leq d+1} {\cal G}(H_{\bm g},\psi_{\bm g}^{i,j})$
if and only if
\begin{align*}
 |F|&\leq \sum_{1\leq i<j\leq d+1}r_{\bm g}^{i,j}(F) \\
&=\sum_{X:\text{ component of } F}\left\{{d+1\choose 2}|V(X)|-{d+1\choose 2}+\sum_{1\leq i<j\leq d+1}\alpha_{\bm g}^{i,j}(X)\right\}
\end{align*}
for every $F\subseteq E(H_{\bm g})$. It is routine to check that this condition can be simplified to \textbf{(ii)}.

To complete the proof we now prove \textbf{(i)}$\Rightarrow$\textbf{(ii)} and then \textbf{(iii)}$\Rightarrow$\textbf{(i)}.
By Proposition~\ref{prop:zero_loop2},
every loop not in $H_{\bm g}$ is a zero loop in $O_{\bm g}(H,\psi,\tilde{\bb})$.
Thus, \textbf{(i)} is equivalent to
\begin{description}
\item[(i')] $O_{\bm g}(H_{\bm g},\psi,\tilde{\bb})$ is row independent.
\end{description}

For $F\subseteq E(H_{\bm g})$, let $I_F=\{(i,j)\mid 1\leq i<j\leq d+1, \alpha_{\bm g}^{i,j}(F)=0\}$.
To show that \textbf{(i')} implies \textbf{(ii)} we show
\begin{equation}
\label{eq:i_ii}
\dim {\rm ker}\ O_{\bm g}(H[F],\psi,\bb)\geq |I_F|.
\end{equation}
This in turn implies that for the row independence of $O_{\bm g}(H_{\bm g},\psi,\tilde{\bb})$,
we need $|F|\leq {d+1\choose 2}|V(F)|-|I_F|$, that is, condition \textbf{(ii)}.

To see (\ref{eq:i_ii}), for each $(i,j)\in I_F$, we  define
$\tilde{{\bm m}}_{i,j}:V(F)\rightarrow \mathbb{R}^{d+1\choose 2}$ as follows.
Since $F$ contains no negative cycle in $(H_{\bm g},\psi_{\bm g}^{i,j})$,
there is a partition of $V(F)$ into two sets $X^{i,j}, Y^{i,j}$
(one of which may be empty) such that $\psi_{\bm g}^{i,j}(\te)=-1$ if and only if
$\te$ joins a vertex in $X^{i,j}$ with a vertex in $Y^{i,j}$.
(To see this, consider the gain graph obtained from $(H[F],\psi_{\bm g}^{i,j})$ by contracting
every edge having the identity label. Since every cycle in $F$ is positive, the resulting graph
is bipartite, and the resulting two classes of the vertex set indicate the desired
bipartition $\{X^{i,j}, Y^{i,j}\}$ of $V(F)$. )
Define $\tilde{{\bm m}}_{i,j}:V(F)\rightarrow \mathbb{R}^{d+1\choose 2}$ by
\begin{equation*}
\tilde{{\bm m}}_{i,j}(\tv)=\begin{cases}
{\bf e}_i\wedge {\bf e}_j & \text{ if $\tv\in X^{i,j}$} \\
-{\bf e}_i\wedge {\bf e}_j & \text{ if $\tv\in Y^{i,j}$} \\
\end{cases} \qquad (\tv\in V(F))
\end{equation*}
where $\{{\bf e}_1, {\bf e}_2,\dots, {\bf e}_{d+1}\}$ is the standard basis
of $\mathbb{R}^{d+1}$.

From the definition of $\psi_{\bm g}^{i,j}$, for each $\te=(\tu,\tv)\in F$, we have
\[
 \tilde{\bmm}_{i,j}(\tu)-\hat{\tau}_{\bm g}^{(2)}(\psi_{\te})\tilde{\bmm}_{i,j}(\tv)=
\pm({\bf e}_i\wedge {\bf e}_j-(\psi_{\bm g}^{i,j}(\te))^2{\bf e}_i\wedge {\bf e}_j)=0.
\]
Thus,
$\langle \tilde{b}(\te), \tilde{\bmm}_{i,j}(\tu)-\hat{\tau}_{\bm g}^{(2)}(\psi_{\te})\tilde{\bmm}_{i,j}(\tv)\rangle=0$
for every $\te\in F$.
This implies (according to (\ref{eq:body_j_inf_quot})) that
$\tilde{{\bm m}}_{i,j}$ is in the kernel of $O_{\bm g}(H[F],\psi,\tilde{\bb})$.
Since $\{\tilde{\bmm}_{i,j}\mid (i,j)\in I_F\}$ is linearly independent, we verified (\ref{eq:i_ii}).

Finally, let us prove \textbf{(iii)}$\Rightarrow$\textbf{(i')}.
Suppose that $E(H)$ can be decomposed into ${d+1\choose 2}$ subgraphs $\{H_{i,j}\mid 1\leq i<j\leq d+1\}$,
as specified in the statement.

We first consider the case where $L=\emptyset$ (i.e., $\Gamma$  acts freely on $E(G)$).
Based on the decomposition, we define $\tilde{\bb}':E(H_{\bm g})\rightarrow Gr(2,d+1)$ by
\begin{equation}
\label{eq:iii_i1}
\tilde{\bb}'(\te)={\bf e}_i\wedge {\bf e}_j \qquad (\te\in E(H_{i,j})).
\end{equation}
Then observe that by changing the column and the row orderings,
$O_{\bm g}(H_{\bm g},\psi,\tilde{\bb})$ is in the following  block-diagonalized form:
\begin{equation}
\label{eq:iii_i2}
\begin{array}{c|c|c|c|c|}
\multicolumn{1}{c}{} & \multicolumn{1}{c}{(1,2)} & \multicolumn{1}{c}{(1,3)} & \multicolumn{1}{c}{\dots} & \multicolumn{1}{c}{(d,d+1)}
\\ \cline{2-5}
E(H_{1,2}) & I(H_{1,2},\psi_{\bm g}^{1,2}) & \multicolumn{2}{c}{ } &  \\ \cline{2-3}
E(H_{1,3}) &   & I(H_{1,3},\psi_{\bm g}^{1,3}) & \multicolumn{1}{c}{ }  & \hsymb{0} \\ \cline{3-3}
\vdots & \multicolumn{2}{c}{}  & \multicolumn{1}{c}{\ \ \ \ddots\ \ \ } & \\ \cline{5-5}
E(H_{d,d+1}) & \multicolumn{1}{c}{ \hsymb{0} }  & \multicolumn{1}{c}{} & & I(H_{d,d+1},\psi_{\bm g}^{d,d+1}) \\ \cline{2-5}
\end{array}
\end{equation}
where each block $I(H_{i,j},\psi_{\bm g}^{i,j})$ is a matrix representation of
${\cal G}(H_{i,j},\psi_{\bm g}^{i,j})$ (cf.~Section~\ref{subsec:dowling}).
Since  $E(H_{i,j})$ is independent in ${\cal G}(H_{i,j},\psi_{\bm g}^{i,j})$,
$O_{\bm g}(H_{\bm g},\psi,\tilde{\bb}')$ is row independent.

If $L\neq \emptyset$,
we have to be careful, since $\tilde{\bb}'(\te)$ of $\te\in L$ has to be a 2-extensor
of the form $\hat{p}\wedge \hat{\tau}(\psi_{\te})\hat{p}$
for some $\hat{p}\in \mathbb{R}^{d+1}$ by (\ref{eq:not_free_loop}).
We claim the following.
\begin{claim}
Let $\te$ be a loop in $E(H_{i,j})\cap L$ and  let $\hat{p}={\bf e}_i+{\bf e}_j\in \mathbb{R}^{d+1}$.
Then
$\left(I_{d+1\choose 2}-(\hat{\tau}^{(2)}(\psi_{\te}))^{-1}\right)(\hat{p}\wedge \hat{\tau}(\psi_{\te})\hat{p})$
is a scalar multiple of  ${\bf e}_i\wedge {\bf e}_j$.
\end{claim}
\begin{proof}
Since $\te$ is in $H_{\bm g}$, $\rho_{\bm g}(\psi_{\te})\neq -1$ holds, and hence $\rho_{\bm g}(\psi_{\te})=1$.

Also, we must have $\psi_{\bm g}^{i,j}(\te)=-1$, for
otherwise $E(H_{i,j})$ contains a loop with identity label, a contradiction.
Recall that $\hat{\tau}(\psi_{\te})$ is a diagonal matrix with entries in $\{-1,1\}$.
Let $k_i\in \{-1,1\}$ be the value of the $i$-th diagonal entry.
Then observe that $\tau_{\bm g}^{i,j}(\psi_{\te})=k_ik_j$.
Therefore, by $\tau_{\bm g}^{i,j}(\psi_{\te})=\psi_{\bm g}^{i,j}(\te)=-1$, we obtain $k_ik_j=-1$.

Since $\hat{p}\wedge \hat{\tau}(\psi_{\te})\hat{p}=({\bf e}_i+{\bf e}_j)\wedge (k_i{\bf e}_i+k_j{\bf e}_j)=
(k_j-k_i){\bf e}_i\wedge {\bf e}_j$,
we have
$\left(I_{d+1\choose 2}-(\hat{\tau}^{(2)}_{\bm g}(\psi_{\te}))^{-1}\right)(\hat{p}\wedge \hat{\tau}(\psi_{\te})\hat{p})
=\left(I_{d+1\choose 2}-(\hat{\tau}^{(2)}_{\bm g}(\psi_{\te}))^{-1}\right)
\left((k_j-k_i){\bf e}_i\wedge {\bf e}_j\right)
=(1-k_i^{-1}k_j^{-1})(k_j-k_i)({\bf e}_i\wedge {\bf e}_j)$.
By $k_ik_j=-1$, $(1-k_i^{-1}k_j^{-1})(k_i-k_j)$ is nonzero, which implies the statement.
\end{proof}
Following this claim, we  define $\tilde{\bb}':E(H)\rightarrow Gr(2,d+1)$ by
\begin{equation*}
\tilde{\bb}'(\te)=\begin{cases}
{\bf e}_i\wedge {\bf e}_j & \text{if $\te\notin L$} \\
({\bf e}_i+{\bf e}_j)\wedge \hat{\tau}(\psi_{\te})({\bf e}_i+{\bf e}_j) & \text{if $\te\in L$}
\end{cases}
\qquad (\te\in E(H_{i,j})).
\end{equation*}
Then $O_{\bm g}(H,\psi,\tilde{\bb}')$ is block-diagonalized in the form of (\ref{eq:iii_i2}),
and ${\rm rank}\ O_{\bm g}(H_{\bm g},\psi,\tilde{\bb}')=|E(H_{\bm g})|$.
In other words {\bf (i')} holds.
%Since every loop not contained in $H_{\bm g}$ is a zero loop, we conclude that
% ${\rm rank}\ O_{\bm g}(H,\psi,\tilde{\bb})=|E(H_{\bm g})|$.
\end{proof}

Note that the dimension of the space of $\rho_{\bm g}$-symmetric trivial infinitesimal motions is equal to
\begin{equation*}
\frac{1}{|\Gamma|}\sum_{\gamma \in \Gamma} {\rm Trace}(\hat{\tau}^{(2)}_{\bm g}(\gamma)).
\end{equation*}

%{d+1\choose 2}-\dim{\rm span}\{{\rm image}(I_{d+1\choose 2}-\hat{\tau}^{(2)}_{\bm g}(\gamma))\mid \gamma \in \Gamma\}.
%\end{equation}

\begin{corollary}
\label{cor:body_bar}
Let $\Gamma=\mathbb{Z}/2\mathbb{Z}\times \dots \times \mathbb{Z}/2\mathbb{Z}$,
 $\tau:\Gamma\rightarrow O(\mathbb{R}^d)$ be a  faithful representation,
$(G,\bb)$ be a $\Gamma$-regular body-bar framework, and
$(H,\psi)$ be the corresponding quotient $\Gamma$-gain graph.
Then the following are equivalent.
\begin{itemize}
\item $(G,\bb)$ is infinitesimally rigid;
\item  for every ${\bm g}\in \Gamma$, $H$ contains a spanning subgraph
$H_{\bm g}$ such that
\begin{description}
\item[(1)] $H_{\bm g}$ contains no zero loop, i.e., a loop $\tilde{e}\in L$ with $\rho_{\bm g}(\psi_{\te})=-1$;
\item[(2)] $|E(H_{\bm g})|={d+1\choose 2}|V(H_{\bm g})|-
\frac{1}{|\Gamma|}\sum_{\gamma \in \Gamma} {\rm Trace}(\hat{\tau}^{(2)}_{\bm g}(\gamma))$;
\item[(3)] for every $F\subseteq E(H_{\bm g})$,  $|F|\leq {d+1\choose 2}|V(F)|-{d+1\choose 2}+\sum_{1\leq i<j\leq d+1}\alpha_{\bm g}^{i,j}(F)$,
where $\alpha_{\bm g}^{i,j}$ is defined as in (\ref{eq:alpha}).
\end{description}
\item for every ${\bm g}$, $H$ contains a subgraph $H_{\bm g}$ satisfying {\rm \textbf{(1)}} and {\rm \textbf{(2)}} that contains ${d+1\choose 2}$ edge-disjoint subgraphs $H_{1,2},\dots, H_{d,d+1}$ such that for every $1\leq i<j\leq d+1$ every connected component of $(H_{i,j},\psi_{\bm g}^{i,j})$ contains no cycle or just one cycle, which is negative. 
\end{itemize}
\end{corollary}

As we will see in the following examples, checking condition \textbf{(ii)} of Theorem~\ref{thm:comb_body}
or condition \textbf{(3)} of Corollary~\ref{cor:body_bar} by hand is applicable only for very small graphs
and the characterization in terms of the counting conditions in \textbf{(ii)} or \textbf{(3)} do not provide a polynomial size certificate that a framework is infinitesimally rigid.
Instead, one can use the characterization in terms of graph decompositions given in \textbf{(iii)} to give a polynomial size certificate for an infinitesimally rigid framework.
In general, these conditions can be checked in $O(|V(H)|^{5/2}|E(H)|)$ time by a matroid union algorithm~\cite{cunningham},
where the independence testing in each matroid can be done in $O(|V(H)|)$ time.
Developing a faster algorithm is left as an open problem.

\subsection{Examples}
\label{subsec:body_bar_examples}
Let us illustrate Theorem~\ref{thm:comb_body} and Corollary~\ref{cor:body_bar} via two examples. First, consider the  $\mathcal{C}_s$-generic Stewart platform $(G,\bb)$  from Section~\ref{sec:exambb},
where $\mathcal{C}_s=\{id,s\}$ and $id$ and $s$ are identified with $0$ and $1$, respectively.
Using Corollary~\ref{cor:body_bar}, we show that $(G,\bb)$ is infinitesimally flexible.

From the $\mathcal{C}_s$-gain graph $(H,\psi)$ of $(G,\bb)$,
we first construct the $\mathcal{C}_s$-gain graphs $(H_0,\psi)$ and $(H_1,\psi)$ which are obtained from $(H,\psi)$ by removing the loops $\tilde{e}\in L$ with $\rho_0(\psi_{\tilde{e}})=-1$ and $\rho_1(\psi_{\tilde{e}})=-1$, respectively (as defined in Theorem~\ref{thm:comb_body}). See also Figure~\ref{fig:bbpicthmill}.
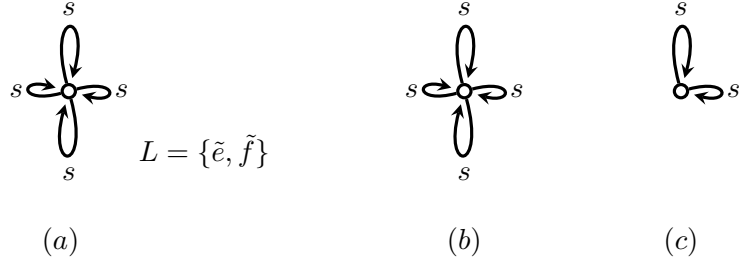
\begin{figure}[htp]
\begin{center}
   \begin{tikzpicture}[very thick,scale=1]
\tikzstyle{every node}=[circle, draw=black, fill=white, inner sep=0pt, minimum width=5pt];
       \path (0,-1) node (p1)  {} ;

 \path
(p1) edge [loop above,->, >=stealth,shorten >=2pt,looseness=46] (p1);
 \path
(p1) edge [loop left,->, >=stealth,shorten >=2pt,looseness=26] (p1);
 \path
(p1) edge [loop below,->, >=stealth,shorten >=2pt,looseness=46] (p1);
 \path
(p1) edge [loop right,->, >=stealth,shorten >=2pt,looseness=26] (p1);

\node [draw=white, fill=white,rectangle] (c) at (-0.7,-1) {$s$};
\node [draw=white, fill=white,rectangle] (c) at (0.7,-1) {$s$};
\node [draw=white, fill=white,rectangle] (c) at (0,0.1) {$s$};
\node [draw=white, fill=white,rectangle] (c) at (0,-2.1) {$s$};

\node [draw=white, fill=white,rectangle] (c) at (1.8,-1.8) {$L=\{\tilde{e},\tilde{f}\}$};

\node [draw=white, fill=white,rectangle] (a) at (0.5,-3) {$(a) \, \qquad \quad$};
       \end{tikzpicture}
     \hspace{1.5cm}
         \begin{tikzpicture}[very thick,scale=1]
\tikzstyle{every node}=[circle, draw=black, fill=white, inner sep=0pt, minimum width=5pt];
       \path (0,-1) node (p1)  {} ;
       %\path (0.9,-0.3) node (p8) [label =  above: $u$] {} ;

 \path
(p1) edge [loop above,->, >=stealth,shorten >=2pt,looseness=46] (p1);
 \path
(p1) edge [loop left,->, >=stealth,shorten >=2pt,looseness=26] (p1);
 \path
(p1) edge [loop below,->, >=stealth,shorten >=2pt,looseness=46] (p1);
 \path
(p1) edge [loop right,->, >=stealth,shorten >=2pt,looseness=26] (p1);

\node [draw=white, fill=white,rectangle] (c) at (-0.7,-1) {$s$};
\node [draw=white, fill=white,rectangle] (c) at (0.7,-1) {$s$};
\node [draw=white, fill=white,rectangle] (c) at (0,0.1) {$s$};
\node [draw=white, fill=white,rectangle] (c) at (0,-2.1) {$s$};

\node [draw=white, fill=white,rectangle] (a) at (0,-3) {$(b)$};
       \end{tikzpicture}
 \hspace{1.5cm}
         \begin{tikzpicture}[very thick,scale=1]
\tikzstyle{every node}=[circle, draw=black, fill=white, inner sep=0pt, minimum width=5pt];
       \path (0,-1) node (p1)  {} ;
       %\path (0.9,-0.3) node (p8) [label =  above: $u$] {} ;

 \path
(p1) edge [loop above,->, >=stealth,shorten >=2pt,looseness=46] (p1);
 \path
(p1) edge [loop right,->, >=stealth,shorten >=2pt,looseness=26] (p1);

\node [draw=white, fill=white,rectangle] (c) at (0.7,-1) {$s$};
\node [draw=white, fill=white,rectangle] (c) at (0,0.1) {$s$};

\node [draw=white, fill=white,rectangle] (a) at (0,-3) {$(c)$};
       \end{tikzpicture}
\end{center}
\vspace{-0.3cm}
\caption{The quotient gain graph $(H,\psi)$ of the body-bar framework in Section~\ref{sec:exambb} (a) and the  gain graphs $(H_{0},\psi)$ (b) and $(H_{1},\psi)$ (c).}
\label{fig:bbpicthmill}
\end{figure}

Then we have  $$|E(H_{0})|=4> 3=6|V(H_{0})|-
\frac{1}{|\mathcal{C}_s|}\sum_{\gamma \in \mathcal{C}_s} {\rm Trace}(\hat{\tau}^{(2)}_{\bm 0}(\gamma))
$$
since $\hat{\tau}^{(2)}_0(id)=I_6$ and $$\hat{\tau}^{(2)}_0(s)=\left(\begin{array}{c c c c c c} 1 & 0 & 0 & 0 & 0 & 0\\ 0 & -1 & 0 & 0 & 0& 0\\ 0& 0 & 1 & 0& 0 & 0\\ 0& 0 & 0& -1 & 0 & 0 \\ 0& 0 & 0 & 0& 1 & 0\\ 0& 0 & 0& 0 & 0 & -1 \end{array}\right).$$ (Recall the definition of  $\hat{\tau}^{(2)}_0(s)=\rho_0(s)\cdot \hat{\tau}^{(2)}(s)=\hat{\tau}^{(2)}(s)$ from Section~\ref{sec:exambb}.)
Similarly, we have $$|E(H_{1})|=2<3=6|V(H_{1})|-
\frac{1}{|\mathcal{C}_s|}\sum_{\gamma \in \mathcal{C}_s} {\rm Trace}(\hat{\tau}^{(2)}_{1}(\gamma)).$$ Thus, condition \textbf{(2)} in Corollary~\ref{cor:body_bar} is violated for $H_1$, and hence  $(G,\bb)$ has a $\rho_1$-symmetric (or anti-symmetric) infinitesimal flex.

As a second example, let us consider a $\mathcal{C}_2$-generic body-bar realization $(G,\bb)$ of the same multigraph $G$ (as shown in Figure~\ref{fig:bbc2real} (a)), where $\mathcal{C}_2=\{id, C_2\}$ describes half-turn symmetry and $id$ and $C_2$ are identified with $0$ and $1$ in $\mathbb{Z}/2\mathbb{Z}$, respectively. Recall 
%from Table~\ref{tabc2} 
that the  group $\mathcal{C}_2$ has two non-equivalent  irreducible representations which are denoted by $\rho_0$ and $\rho_1$.

\vspace{0.2cm}
\begin{figure}[htp]
\begin{center}
\begin{tikzpicture}[very thick,scale=1]
\tikzstyle{every node}=[circle, draw=black, fill=white, inner sep=0pt, minimum width=5pt];
\filldraw[fill=black!20!white, draw=black, thin](0,-1.8)ellipse(1.2cm and 0.35cm);
\filldraw[fill=black!20!white, draw=black, thin](0,-0.2)ellipse(1.2cm and 0.35cm);

\node [circle, draw=black!20!white, shade, ball color=black!40!white, inner sep=0pt, minimum width=7pt](p1) at (-0.3,-1.95) {};
\node [circle, draw=black!20!white, shade, ball color=black!40!white, inner sep=0pt, minimum width=7pt](p2) at (-0.78,-1.85) {};
\node [circle, draw=black!20!white, shade, ball color=black!40!white, inner sep=0pt, minimum width=7pt](p3) at (0,-1.6) {};
\node [circle, draw=black!20!white, shade, ball color=black!40!white, inner sep=0pt, minimum width=7pt](p4) at (0.4,-1.8) {};
\node [circle, draw=black!20!white, shade, ball color=black!40!white, inner sep=0pt, minimum width=7pt](p5) at (0.7,-1.7) {};
\node [circle, draw=black!20!white, shade, ball color=black!40!white, inner sep=0pt, minimum width=7pt](p6) at (1,-1.8) {};

\node [circle, draw=black!20!white, shade, ball color=black!40!white, inner sep=0pt, minimum width=7pt](t1) at (-0.3,-0.05) {};
\node [circle, draw=black!20!white, shade, ball color=black!40!white, inner sep=0pt, minimum width=7pt](t2) at (-0.78,-0.15) {};
\node [circle, draw=black!20!white, shade, ball color=black!40!white, inner sep=0pt, minimum width=7pt](t3) at (0,-0.4) {};
\node [circle, draw=black!20!white, shade, ball color=black!40!white, inner sep=0pt, minimum width=7pt](t4) at (0.4,-0.2) {};
\node [circle, draw=black!20!white, shade, ball color=black!40!white, inner sep=0pt, minimum width=7pt](t5) at (0.7,-0.3) {};
\node [circle, draw=black!20!white, shade, ball color=black!40!white, inner sep=0pt, minimum width=7pt](t6) at (1,-0.2) {};

\node [draw=white, fill=white,rectangle] (a) at (-0.05,-0.8) {$h$};
\node [draw=white, fill=white,rectangle] (a) at (0.6,-0.8) {$k$};
\draw(p1)--(t1);
\draw(p2)--(t2);
\draw(p3)--(t4);
\draw(p4)--(t3);
\draw(p5)--(t6);
\draw(p6)--(t5);
\node [draw=white, fill=white] (a) at (0,-3) {$(a)$};

\node [draw=white, fill=white,rectangle] (a) at (-1,-0.8) {$e$};
\node [draw=white, fill=white,rectangle] (a) at (-0.5,-0.8) {$f$};

%\node [draw=white, fill=white,rectangle] (b) at (-0.7,0.36) {$p_{e,u}$};
%\node [draw=white, fill=white,rectangle] (c) at (-0.9,-2.3) {$p_{e,v}$};
\node [draw=white, fill=white,rectangle] (c) at (0,0.5) {$u$};
\node [draw=white, fill=white,rectangle] (c) at (0,-2.45) {$v$};
%\node [draw=white, fill=white,rectangle] (c) at (-2,-0.9) {$(G,\bb)$};
\draw[dashed, thin](-1.8,-1)--(2,-1);
\node [draw=white, fill=white,rectangle] (c) at (1.8,-1.25) {$C_2$};

\end{tikzpicture}
\hspace{0.8cm}
   \begin{tikzpicture}[very thick,scale=1]
\tikzstyle{every node}=[circle, draw=black, fill=white, inner sep=0pt, minimum width=5pt];
       \path (0,-1) node (p1)  {} ;
       %\path (0.9,-0.3) node (p8) [label =  above: $u$] {} ;

 \path
(p1) edge [loop above,->, >=stealth,shorten >=2pt,looseness=46] (p1);
 \path
(p1) edge [loop left,->, >=stealth,shorten >=2pt,looseness=26] (p1);
 \path
(p1) edge [loop below,->, >=stealth,shorten >=2pt,looseness=46] (p1);
 \path
(p1) edge [loop right,->, >=stealth,shorten >=2pt,looseness=26] (p1);

%\node [draw=white, fill=white,rectangle] (c) at (-0.35,-0.65) {$\tilde{e}$};
%\node [draw=white, fill=white,rectangle] (c) at (-0.3,-1.5) {$\tilde{f}$};

\node [draw=white, fill=white,rectangle] (c) at (-0.9,-1) {$C_2$};
\node [draw=white, fill=white,rectangle] (c) at (0.9,-1) {$C_2$};
\node [draw=white, fill=white,rectangle] (c) at (0,0.1) {$C_2$};
\node [draw=white, fill=white,rectangle] (c) at (0,-2.15) {$C_2$};

\node [draw=white, fill=white,rectangle] (c) at (1.8,-1.8) {$L=\{\tilde{e},\tilde{f}\}$};

\node [draw=white, fill=white] (a) at (0,-3) {$(b)$};
       \end{tikzpicture}
      \hspace{0.8cm}
         \begin{tikzpicture}[very thick,scale=1]
\tikzstyle{every node}=[circle, draw=black, fill=white, inner sep=0pt, minimum width=5pt];
       \path (0,-1) node (p1)  {} ;
       %\path (0.9,-0.3) node (p8) [label =  above: $u$] {} ;

 \path
(p1) edge [loop above,->, >=stealth,shorten >=2pt,looseness=46] (p1);
 \path
(p1) edge [loop left,->, >=stealth,shorten >=2pt,looseness=26] (p1);
 \path
(p1) edge [loop below,->, >=stealth,shorten >=2pt,looseness=46] (p1);
 \path
(p1) edge [loop right,->, >=stealth,shorten >=2pt,looseness=26] (p1);

\node [draw=white, fill=white,rectangle] (c) at (-0.9,-1) {$C_2$};
\node [draw=white, fill=white,rectangle] (c) at (0.9,-1) {$C_2$};
\node [draw=white, fill=white,rectangle] (c) at (0,0.1) {$C_2$};
\node [draw=white, fill=white,rectangle] (c) at (0,-2.15) {$C_2$};

\node [draw=white, fill=white,rectangle] (a) at (0,-3) {$(c)$};
       \end{tikzpicture}
 \hspace{0.8cm}
         \begin{tikzpicture}[very thick,scale=1]
\tikzstyle{every node}=[circle, draw=black, fill=white, inner sep=0pt, minimum width=5pt];
       \path (0,-1) node (p1)  {} ;
       %\path (0.9,-0.3) node (p8) [label =  above: $u$] {} ;

 \path
(p1) edge [loop above,->, >=stealth,shorten >=2pt,looseness=46] (p1);
 \path
(p1) edge [loop right,->, >=stealth,shorten >=2pt,looseness=26] (p1);

\node [draw=white, fill=white,rectangle] (c) at (0.9,-1) {$C_2$};
\node [draw=white, fill=white,rectangle] (c) at (0,0.1) {$C_2$};

\node [draw=white, fill=white,rectangle] (a) at (0,-3) {$(d)$};
       \end{tikzpicture}
\end{center}
\vspace{-0.3cm}
\caption{A Stewart platform with half-turn symmetry (a), its quotient gain graph $(H,\psi)$ (b) and the induced gain graphs $(H_0,\psi)$ (c) and $(H_1,\psi)$ (d). }
\label{fig:bbc2real}
\end{figure}
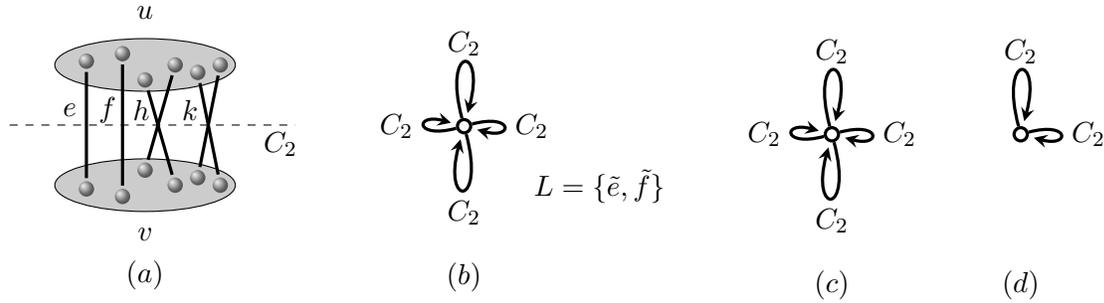

Suppose that the half-turn axis of $C_2$ is the $x$-axis, that is,  $\hat{\tau}(C_2)= \left(\begin{array}{c c c c} 1 & 0 & 0 & 0\\ 0 & -1 & 0 & 0\\ 0& 0 & -1 & 0\\ 0& 0 & 0&1\end{array}\right)$. Then we have
\begin{equation}\label{eq:tauhat} \hat{\tau}^{(2)}_g(C_2)=\rho_g(C_2)\cdot \hat{\tau}^{(2)}(C_2)=\rho_g(C_2)\cdot\left(\begin{array}{c c c c c c} -1 & 0 & 0 & 0 & 0 & 0\\ 0 & -1 & 0 & 0 & 0& 0\\ 0& 0 & 1 & 0& 0 & 0\\ 0& 0 & 0& 1 & 0 & 0 \\ 0& 0 & 0 & 0& -1 & 0\\ 0& 0 & 0& 0 & 0 & -1 \end{array}\right),\end{equation}
where $\rho_g(C_2)=1$ for $g=0$ and $\rho_g(C_2)=-1$ for $g=1$.

Conditions \textbf{(1)} and \textbf{(2)} of Corollary~\ref{cor:body_bar} are then clearly satisfied, since we have
$$|E(H_{0})|=4=6|V(H_{0})|-
\frac{1}{|\mathcal{C}_2|}\sum_{\gamma \in \mathcal{C}_2}
{\rm Trace}(\hat{\tau}^{(2)}_{0}(\gamma)).$$
and
$$|E(H_{1})|=2=6|V(H_{1})|-
\frac{1}{|\mathcal{C}_2|}\sum_{\gamma \in \mathcal{C}_2} {\rm Trace}(\hat{\tau}^{(2)}_{1}(\gamma)).$$
So let us check condition \textbf{(3)} of Corollary~\ref{cor:body_bar}. First, we consider $H_0$ shown in Figure~\ref{fig:bbc2real}(c). Let $F$ be a subset of $E(H_0)$ which consists of a single loop, say $F=\{\te\}$ (where $\psi(\te)=C_2$). Then
$$\psi^{i,j}_0(\te)=\tau^{i,j}_0(\psi(\te))=\tau^{i,j}_0(C_2),$$
and hence, by (\ref{eq:tauhat}), $\psi^{i,j}_0(\te)=-1$ for $(i,j)=(1,2),(1,3),(2,4),(3,4)$ and $\psi^{i,j}_0(\te)=1$ for $(i,j)=(1,4),(2,3)$. Thus, by (\ref{eq:alpha}), $\sum_{1\leq i < j \leq 6}\alpha^{i,j}_0(F)=1+1+0+0+1+1=4$, and hence
$$|F|=1<4= 6|V(F)|-6+\sum_{1\leq i < j \leq 6}\alpha^{i,j}_0(F).$$
For the other subsets of $E(H_0)$, condition \textbf{(3)} of Corollary~\ref{cor:body_bar} is verified analogously.

Finally, consider $H_1$ shown in Figure~\ref{fig:bbc2real}(d). Let $F$ be a subset of $E(H_1)$ which consists of a single loop, say $F=\{\tilde{h}\}$ (where $\psi(\tilde{h})=C_2$). Then
$$\psi^{i,j}_1(\tilde{h})=\tau^{i,j}_1(\psi(\tilde{h}))=\tau^{i,j}_1(C_2),$$
and hence, by (\ref{eq:tauhat}), $\psi^{i,j}_1(\tilde{h})=1$ for $(i,j)=(1,2),(1,3),(2,4),(3,4)$ and $\psi^{i,j}_1(\tilde{h})=-1$ for $(i,j)=(1,4),(2,3)$. Thus, by (\ref{eq:alpha}), we have
$$|F|=1<2= 6|V(F)|-6+\sum_{1\leq i < j \leq 6}\alpha^{i,j}_1(F).$$
For the other subsets of $E(H_1)$, condition \textbf{(3)} of Corollary~\ref{cor:body_bar} is again verified analogously.

Therefore, we may conclude that $\mathcal{C}_2$-generic body-bar realizations of $G$ (such as the one in Figure~\ref{fig:bbc2real}(a)) are infinitesimally rigid (isostatic).

%recall from Section~\ref{sec:exambb} that

%%%%%%%%%%%%%%%%%%%%%%%%%%%%%%%%%%%%%%%%%%%%%%%%%%%%%%%%%%%%%%%%%%%%%%%%%%%%%%%%%%%%%%%%%%%%%%%%%%%%%%%%%%%%%%%%%%%%%%%%%%%

\section{Body-hinge frameworks}\label{sec:hinge}

A body-hinge framework is a structural model consisting of rigid bodies which are pairwise connected by hinges as shown in Figure~\ref{fig:bh}(a).
A body-hinge framework can again be regarded as a special case of a bar-joint framework
by replacing each body by a complete framework with sufficiently many joints, and
all the theory developed so far can be applied to this model.

Of particular importance for applications (e.g., for rigidity and flexibility analyses of biomolecules or robotic linkages) are $3$-dimensional body-hinge frameworks. Since a hinge removes $5$ of the $6$ relative degrees of freedom between a pair of rigid bodies in $3$-space, a $3$-dimensional body-hinge
framework can be modeled as a special case of a body-bar framework by replacing each hinge with 5 independent bars, each intersecting the hinge line (see Figure~\ref{fig:bh}(a)).

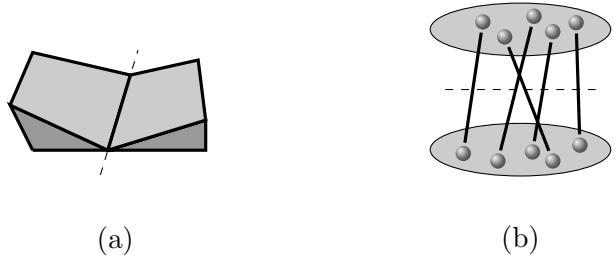
\begin{figure}[htp]
\begin{center}
\begin{tikzpicture}[very thick,scale=1]
\tikzstyle{every node}=[circle, draw=black, fill=white, inner sep=0pt, minimum width=5pt];
\filldraw[fill=black!40!white, draw=black]
    (0,0) -- (1,0) -- (-0.3,0.6) -- (0,0) ;
    \filldraw[fill=black!20!white, draw=black]
    (-0.3,0.6) -- (1,0) -- (1.3,1) -- (0,1.3) -- (-0.3,0.6);

     \filldraw[fill=black!40!white, draw=black]
     (1,0) -- (2.3,0) -- (2.3,0.4) -- (1,0);
     \filldraw[fill=black!20!white, draw=black]
    (1,0) -- (1.3,1) -- (2.2,1.2) -- (2.3,0.4) -- (1,0);

     \draw[dashed,thin](0.9,-0.33)--(1.4,1.333);
     \node [draw=white, fill=white] (a) at (1.1,-1.2) {(a)};
\end{tikzpicture}
       \hspace{2.7cm}
\begin{tikzpicture}[very thick,scale=1]
\tikzstyle{every node}=[circle, draw=black, fill=white, inner sep=0pt, minimum width=5pt];

\filldraw[fill=black!20!white, draw=black, thin](0,-1.8)ellipse(1.2cm and 0.35cm);
\filldraw[fill=black!20!white, draw=black, thin](0,-0.2)ellipse(1.2cm and 0.35cm);

\node [circle, draw=black!20!white, shade, ball color=black!40!white, inner sep=0pt, minimum width=7pt](p1) at (-0.3,-1.95) {};
\node [circle, draw=black!20!white, shade, ball color=black!40!white, inner sep=0pt, minimum width=7pt](p2) at (-0.76,-1.85) {};
%\node [circle, draw=black!20!white, shade, ball color=black!40!white, inner sep=0pt, minimum width=7pt](p3) at (-0.3,-1.63) {};
\node [circle, draw=black!20!white, shade, ball color=black!40!white, inner sep=0pt, minimum width=7pt](p4) at (0.16,-1.83) {};
\node [circle, draw=black!20!white, shade, ball color=black!40!white, inner sep=0pt, minimum width=7pt](p5) at (0.79,-1.74) {};
\node [circle, draw=black!20!white, shade, ball color=black!40!white, inner sep=0pt, minimum width=7pt](p6) at (0.43,-1.95) {};

\node [circle, draw=black!20!white, shade, ball color=black!40!white, inner sep=0pt, minimum width=7pt](t1) at (0.18,-0.025) {};
%\node [circle, draw=black!20!white, shade, ball color=black!40!white, inner sep=0pt, minimum width=7pt](t2) at (-0.78,-0.25) {};
\node [circle, draw=black!20!white, shade, ball color=black!40!white, inner sep=0pt, minimum width=7pt](t3) at (-0.5,-0.1) {};
\node [circle, draw=black!20!white, shade, ball color=black!40!white, inner sep=0pt, minimum width=7pt](t4) at (0.43,-0.23) {};
\node [circle, draw=black!20!white, shade, ball color=black!40!white, inner sep=0pt, minimum width=7pt](t5) at (0.74,-0.11) {};
\node [circle, draw=black!20!white, shade, ball color=black!40!white, inner sep=0pt, minimum width=7pt](t6) at (-0.21,-0.3) {};

\draw(p1)--(t1);
\draw(p2)--(t3);
%\draw(p3)--(t2);
\draw(p4)--(t4);
\draw(p5)--(t5);
\draw(p6)--(t6);
\node [draw=white, fill=white] (a) at (0,-3) {(b)};

\draw[dashed,thin](-1,-1)--(1,-1);
\end{tikzpicture}

\end{center}
\vspace{-0.3cm}
\caption{(a) A $3$-dimensional body-hinge framework consisting of two bodies which are connected by a hinge. (b) In $3$-space, a hinge can be modeled as a set of $5$ independent bars, each intersecting the hinge line. }
\label{fig:bh}
\end{figure}

%Its importance is widely recognized because the rigidity of molecules can be modeled as body-hinge frameworks.

The infinitesimal rigidity of generic body-hinge frameworks in $\mathbb{R}^d$ was
characterized independently by Whiteley \cite{wwmatun,TW1} and Tay \cite{tay89,tay91}. In the following, we will give a symmetric version of  their result by  formulating the infinitesimal rigidity of body-hinge frameworks again in terms of
Pl{\"u}cker coordinates.

We define a {\em body-hinge framework} to be a pair $(G,\bh)$ of an undirected graph $G$ and a hinge-configuration
\begin{align}
\label{eq:bar_confbh}
\begin{split}
\bh:\quad E(G) \quad &\rightarrow Gr(d-1,d+1) \\
e=\{u,v\} &\mapsto \hat{p}_{e,1}\wedge \hat{p}_{e,2}\wedge \dots \wedge \hat{p}_{e,d-1}.
\end{split}
\end{align}
That is, $\bh(e)$ indicates the Pl{\"u}cker coordinates of
a hinge, i.e., a $(d-1)$-dimensional simplex determined by points $p_{e,1},\dots, p_{e,d-1}$
in the bodies of $u$ and $v$.

An infinitesimal motion of a body-hinge framework $(G,\bh)$ is defined as
$\bmm:V(G)\rightarrow \mathbb{R}^{d+1\choose 2}$
satisfying
\begin{equation}
\label{eq:inf_body_hinge}
\bmm(u)-\bmm(v)\in {\rm span}\{\bh(e)\} \qquad \text{for all } \{u,v\}\in E(G).
\end{equation}
Observe that $\bmm$ is an infinitesimal motion if $\bmm(u)=\bmm(v)$ for all $u,v\in V(G)$.
Such a motion is called a trivial motion, and
$(G,\bh)$ is called {\em infinitesimally rigid} if all infinitesimal motions of $(G,\bh)$ are trivial.

For every $e\in E(G)$, let us prepare $({d+1\choose 2}-1)$ copies of $e$,
denoted by $e_1,\dots, e_{{d+1\choose 2}-1}$; the set of all copied edges we denote by $({d+1\choose 2}-1)E(G)$.
Also, let $({d+1\choose 2}-1)G=(V(G), ({d+1\choose 2}-1)E(G))$.

For the hinge-configuration $\bh$, we take $\bb:({d+1 \choose 2}-1)E(G)\rightarrow Gr(2,d+1)$ so that
$\{\bb(e_i)\mid 1\leq i\leq {d+1\choose 2}-1\}$ is a basis of the orthogonal complement of
${\rm span}\{\ast \bh(e)\}$.
Then $(G,\bh)$ is infinitesimally rigid if and only if
$(({d+1\choose 2}-1)G,\bb)$ is infinitesimally rigid.
Thus a body-hinge framework $(G,\bh)$ can be regarded as a body-bar framework
$(({d+1\choose 2}-1)G,\bb)$  with the extra condition
that
$\{\bb(e_i)\mid 1\leq i\leq {d+1\choose 2}-1\}$ is a basis of the orthogonal complement of
a one-dimensional space spanned by $\ast\bh(e)$ for each $e\in E(G)$.

Now let us introduce $\Gamma$-symmetric body-hinge frameworks.
Suppose $\Gamma$ is a group with $\tau:\Gamma\rightarrow O(\mathbb{R}^d)$.
We say that a body-hinge framework $(G,\bh)$ is {\em $\Gamma$-symmetric} (with respect to $\tau$ and
$\theta:\Gamma\rightarrow {\rm Aut}(G)$) if
$G$ is $\Gamma$-symmetric with respect to $\theta$ and
%\[
 %\hat{\tau}(\gamma)\hat{p}_{e,i}=\hat{p}_{\theta(\gamma) e,i} \qquad \text{for all } \gamma\in \Gamma, e\in E(G), 1\leq i\leq {d+1\choose 2}.
%\]
%This implies
$$\bh(\theta(\gamma) e)=\hat{\tau}^{(d-1)}(\gamma)\bh(e) \textrm{ for every } e\in E(G) \textrm{ and } \gamma\in \Gamma.$$

It is not difficult to check that if $(G,\bh)$ is $\Gamma$-symmetric and $\theta$  acts freely on $E(G)$,
then there exists a body-bar framework $(({d+1\choose 2}-1)G,\bb)$  so that
$(({d+1\choose 2}-1)G,\bb)$ is $\Gamma$-symmetric (with respect to
$\tau$ and $\theta':\Gamma\rightarrow {\rm Aut}\left({d+1\choose 2}-1)G\right)$, which is obtained from
$\theta$ in an obvious manner).
The framework $(({d+1\choose 2}-1)G,\bb)$ is called a $\Gamma$-symmetric body-bar framework
{\em associated with} $(G,\bh)$.

We say that $(G,\bh)$ is $\Gamma$-regular if
the dimension of the space of infinitesimal motions of $(G,\bh)$ is minimized among all
$\Gamma$-symmetric body-hinge realizations $(G,\bh')$ of $G$.

Also, for a $\Gamma$-gain graph $(H,\psi)$,
$(({d+1\choose 2}-1)H,\psi)$ denotes the $\Gamma$-gain graph obtained from $(H,\psi)$ by replacing
each edge $\te$ by ${d+1\choose 2}-1$ parallel copies $\te_1,\dots, \te_{{d+1\choose 2}-1}$
with $\psi(\te_i)=\psi(\te)$.

\begin{theorem}
\label{thm:body_hinge}
Let $\Gamma=\mathbb{Z}/2\mathbb{Z}\times \dots \times \mathbb{Z}/2\mathbb{Z}$,
$\tau:\Gamma\rightarrow O(\mathbb{R}^d)$ be a faithful representation,
$(G,\bh)$ be a $\Gamma$-regular  body-hinge framework,
and $(H,\psi)$ be the quotient $\Gamma$-gain graph.
Suppose that $\Gamma$  acts freely on the edge set of $G$.
Then the following are equivalent.
\begin{itemize}
\item $(G,\bh)$ is infinitesimally rigid; 
\item for every ${\bm g}\in \Gamma$, $({d+1\choose 2}-1)H$ contains a spanning subgraph
$H_{\bm g}$ satisfying {\rm \textbf{(2)}} and {\rm \textbf{(3)}} of Corollary~\ref{cor:body_bar};
\item for every ${\bm g}\in \Gamma$, $({d+1\choose 2}-1)H$ contains a spanning subgraph
$H_{\bm g}$ satisfying {\rm \textbf{(2)}} of Corollary~\ref{cor:body_bar} that contains ${d+1\choose 2}$ edge-disjoint spanning subgraphs $H_{1,2}, \dots, H_{d,d+1}$ such that each connected component of $(H_{i,j},\psi_{\bm g}^{i,j})$ contains no cycle or just one cycle, which is negative.
\end{itemize}
\end{theorem}
\begin{proof}
Let $(({d+1\choose 2}-1)G,\bb)$ be a $\Gamma$-symmetric body-bar framework
associated with $(G,\bh)$.
It suffices to show that conditions \textbf{(i)}-\textbf{(iii)} of Theorem~\ref{thm:comb_body} are equivalent for
$(({d+1\choose 2}-1)H,\psi,\tilde{\bb})$.
The equivalence of \textbf{(ii)} and \textbf{(iii)} is nothing but a consequence of the matroid union theorem, as we have seen
in the proof of Theorem~\ref{thm:comb_body}.
Also, the proof of Theorem~\ref{thm:comb_body} shows that
\textbf{(i)}$\Rightarrow$\textbf{(ii)} holds for every $\Gamma$-symmetric body-bar framework.
So it suffices to show \textbf{(iii)}$\Rightarrow$\textbf{(i)} for $(({d+1\choose 2}-1)H,\psi,\tilde{\bb})$.

It should be noted that by construction,
\begin{equation}
\label{eq:hinge_condition2}
\text{$\left\{\tilde{\bb}(\te_i)\mid 1\leq i\leq {d+1\choose 2}-1\right\}$ is a basis of the orthogonal
complement of ${\rm span}\{\ast\tilde{\bh}(\te)\}$ }
\end{equation}
for every $\te\in E(H)$.
This implies that $\tilde{\bb}$ may not be $\Gamma$-regular, and
we need to show that the rank does not decrease even if $\tilde{\bb}$ satisfies (\ref{eq:hinge_condition2}).

To see this, suppose that $({d+1\choose 2}-1)H$ can be decomposed into ${d+1\choose 2}$ subgraphs
$H_{1,2},\dots, H_{d,d+1}$, as specified in \textbf{(iii)}.
We define $\tilde{\bb}':E(({d+1\choose 2}-1)H)\rightarrow Gr(2,d+1)$ by
\begin{equation*}
\tilde{\bb}'(\te)={\bf e}_i\wedge {\bf e}_j \qquad (\te\in E(H_{i,j})).
\end{equation*}
Then in the proof of Theorem~\ref{thm:comb_body} we have already shown that
\[
 {\rm rank}\ O_{\bm g}\left(\left({d+1\choose 2}-1\right)H,\psi,\tilde{\bb'}\right)=\left({d+1\choose 2}-1\right)|E(H)|.
\]

On the other hand, let us define $\tilde{\bh}':E(H)\rightarrow Gr(d-1,d+1)$ as follows.
For each $\te\in E(H)$, there is a pair $(a,b)$ such that $H_{a,b}$ does not contain any copy of $\te$.
Let  $\{i_1,\dots, i_{d-1}\}$ be the complement of $\{a,b\}$ among $\{1,2,\dots, d+1\}$,
and let  $\tilde{\bh}'(\te)={\bf e}_{i_1}\wedge \dots \wedge {\bf e}_{i_{d-1}}$.

Observe that every $H_{i,j}$ contains at most one copy of $\te\in E(H)$.
Therefore, $\{\tilde{\bb}'(\te_i)\mid 1\leq i\leq {d+1\choose 2}\}$ is linearly independent.
Moreover, due to the choice of $\tilde{\bh}'$, we have
$\langle \tilde{\bb}'(\te_i), \ast \tilde{\bh}'(\te)\rangle=  \tilde{\bb}'(\te_i)\circ \tilde{\bh}'(\te)= 0$ for every $\te\in E(H_{\bm g})$ and
any copy $\tilde{e}_i$ of $\te$.
Therefore,
$\{\tilde{\bb}'(\te_i)\mid 1\leq i\leq {d+1\choose 2}\}$ is a basis of the orthogonal
complement of ${\rm span}\{\ast \tilde{\bh}'(\te)\}$.

Thus, $(({d+1\choose 2}-1)G,\bb')$ is a body-bar framework associated with $(G,\bh')$.
Since $\bh$ is $\Gamma$-regular, we obtain
${\rm rank}\ O_{\bm g}(({d+1\choose 2}-1)H,\psi,\tilde{\bb})\geq
{\rm rank}\ O_{\bm g}(({d+1\choose 2}-1)H,\psi,\tilde{\bb}')=({d+1\choose 2}-1)|E(H_{\bm g})|$.
Thus \textbf{(i)} holds.
\end{proof}
%
%Let us finally remark the case when $\theta$ does not freely act on $E(G)$.
%Let $(G,\bh)$ be a $\Gamma$-symmetric body-hinge framework
%with respect to $\theta:\Gamma\rightarrow {\rm Aut}(G)$ and $\tau:\Gamma\rightarrow O(\mathbb{R}^d)$ in 3-space.
%Suppose that there is $\gamma\in \Gamma$ such that
%$\theta(\gamma)e=e$ for $e=\{u,v\}\in E(G)$.
%(As $\theat$ freely acts on $V(G)$, $\theta(\gamma)v=u$ and $\theta(\gamma)u=v$.)
%Then we have $\hat{\tau}(\gamma)\hat{p}_{e,1}=\hat{p}_{e,2}$
%and $\hat{\tau}(\gamma)\hat{p}_{e,2}=\hat{p}_{e,1}$.
 %
If the underlying symmetry has small size, then most of the labeling functions $\psi_{\bm g}^{i,j}$ turn out to be identical and the combinatorial conditions of Theorem~\ref{thm:body_hinge} can be significantly simplified.
For example, in Section~\ref{subsec:body_bar_examples} we have seen the exact coordinates of $\hat{\tau}_{\bm g}^{(2)}$  in the case of $\Gamma={\cal C}_s$ or $\Gamma={\cal C}_2$
and  by specializing Theorem~\ref{thm:body_hinge} to these cases one can easily derive the following.
\begin{corollary}
Let $(G,\bh)$ be a ${\cal C}_s$-regular body-hinge framework in $\mathbb{R}^3$, where ${\cal C}_s$ denotes reflection symmetry. Suppose that ${\cal C}_s$ acts freely on the edge set of $G$. Then $(G,\bh)$ is infinitesimally rigid if and only if the quotient gain graph $(H,\psi)$ contains three edge-disjoint spanning trees and three subgraphs such that each connected component contains exactly one cycle, which is negative.
\end{corollary}
\begin{corollary}
Let $(G,\bh)$ be a ${\cal C}_2$-regular body-hinge framework in $\mathbb{R}^3$, where ${\cal C}_2$ denotes half-turn symmetry. Suppose that ${\cal C}_2$ acts freely on the edge set of $G$. Then $(G,\bh)$ is infinitesimally rigid if and only if the quotient gain graph $(H,\psi)$ contains two edge-disjoint spanning trees and four subgraphs such that each connected component contains exactly one cycle, which is negative.
\end{corollary}

\section{Further work and applications} \label{sec:furtherwork}

In Section~\ref{bbblock}, we constructed new symmetry-adapted rigidity matrices to analyze the infinitesimal rigidity properties of symmetric body-bar frameworks with arbitrary Abelian point group symmetries. Each of these `orbit rigidity matrices' corresponds to an irreducible representation of the point group of the given body-bar framework. However, analogously to the situation for bar-joint frameworks (see \cite[Section 7]{schtan}), it remains open how to construct a $\rho_j$-orbit rigidity matrix of a body-bar framework $(G,\bb)$, where $\rho_j$ is an irreducible representation of the point group of $(G,\bb)$ which is of dimension at least 2. Consequently, it is not yet clear how to construct a full set of orbit rigidity matrices for a body-bar framework with a non-Abelian point group.

Furthermore, note that throughout this paper, we restricted attention to the case where the point group $\Gamma$ of a body-bar framework $(G,\bb)$ acts freely on the vertices of $G$ (i.e., on the bodies of $(G,\bb)$). If we allow $\Gamma$ to act non-freely on the bodies of $(G,\bb)$, then the sizes and entries of the orbit rigidity matrices of $(G,\bb)$ need to be adjusted accordingly.

 For example, suppose $(G,\bb)$ is a $3$-dimensional $\mathcal{C}_s$-symmetric body-bar framework, and a vertex $i$ of $G$ is fixed by the reflection $s$ in $\mathcal{C}_s$ (i.e., $\theta(s) (i)=i$). Then $i$ contributes only three columns to each of the two orbit rigidity matrices of $(G,\bb)$, since the body corresponding to $i$ must `lie on the mirror plane of $s$', and hence has only three fully-symmetric degrees of freedom (translations within the mirror and rotations about the axis perpendicular to the mirror) and also only three anti-symmetric degrees of freedom (translations perpendicular to the mirror and rotations about axes within the mirror). 
 
Consequently, in the case where the point group does not act freely on the bodies of the framework, the construction of the orbit rigidity matrices becomes significantly more messy (see also \cite{BSWW,schtan}), although we do not expect any major new difficulties to arise when  making this extension. However, these modifications to the patterns of the orbit rigidity matrices may give rise to substantial new problems in extending the combinatorial results derived  in Section~\ref{subsec:body_chara} to this more general case.

Finally, we remark that as special cases of our results in Sections~\ref{subsec:body_chara} and \ref{sec:hinge}, we obtain combinatorial characterizations of infinitesimally rigid $3$-dimensional body-bar and body-hinge frameworks which are generic with respect to the point groups $\mathcal{C}_2$ or $\mathcal{D}_2$ - the most common symmetry groups found in proteins \cite{dimers}. In large systems such as  proteins, few if any structural components occupy positions of non-trivial site symmetry, and hence useful global conclusions can be drawn from the study of frameworks under the restriction that the point group acts freely on both the vertex and the edge set of the underlying multigraph.
 Therefore, since our results also lay the foundation to design efficient algorithms for testing symmetry-generic infinitesimal rigidity,
we anticipate that 
our work will also be applied to actual proteins and will lead to a better understanding of the behavior and functionality of symmetric proteins such as dimers.

% (Molecular frameworks, named by analogy with chemical structures, have the property that the lines of the hinges attached to each body all pass through a common point in that body.)

\providecommand{\bysame}{\leavevmode\hbox to3em{\hrulefill}\thinspace}
\providecommand{\MR}{\relax\ifhmode\unskip\space\fi MR }
\providecommand{\MRhref}[2]{%
  \href{http://www.ams.org/mathscinet-getitem?mr=#1}{#2}
}
\providecommand{\href}[2]{#2}


\begin{thebibliography}{10}

%\bibitem{altherz}
%S. L. Altmann and P. Herzig, \emph{Point-Group Theory Tables}, Clarendon
%Press, Oxford, 1994

%\bibitem{asiroth}
%L.~Asimov and B.~Roth, \emph{The {R}igidity of {G}raphs}, AMS \textbf{245}
%  (1978), 279--289.

%\bibitem{bishop}
%D.M. Bishop, \emph{Group {T}heory and {C}hemistry}, Clarendon Press, Oxford,
%  1973.

 % \bibitem{bost} C. Borcea and I. Streinu, \emph{Minimally rigid periodic graphs}, Bulletin of the LMS, \textbf{43} (2011), No. 6, 1093--1103.

\bibitem{cfgsw}
R.~Connelly, P.W. Fowler, S.D. Guest, B.~Schulze, and W.~Whiteley, \emph{When
  is a symmetric pin-jointed framework isostatic?}, International Journal of
  Solids and Structures \textbf{46} (2009), 762--773.

\bibitem{cunningham}
W.~Cunningham,
\newblock \emph{Improved bounds for matroid partition and intersection algorithms},
\newblock {SIAM Journal on Computing}, 15 (1986), 948--957.

%\bibitem{FGsymmax}
%P.W. Fowler and S.D. Guest, \emph{A symmetry extension of {M}axwell's rule for
%  rigidity of frames}, International Journal of Solids and Structures
%  \textbf{37} (2000), 1793--1804.

%\bibitem{dowling1973class}
%T.~Dowling,
%\newblock \emph{A class of geometric lattices based on finite groups},
%\newblock J.~Combin.~Theory Ser.~B, \textbf{14} (1973), 61--86.

%\bibitem{graver}J. Graver, \emph{Counting on Frameworks}, Mathematical Association of America, 2001.

%\bibitem{gss}
%J.E. Graver, B.~Servatius, and H.~Servatius, \emph{Combinatorial {R}igidity},
%  Graduate Studies in Mathematics, AMS, 1993.

\bibitem{gsw}
S.D. Guest, B.~Schulze, and W.~Whiteley, \emph{When is a symmetric body-bar
  structure isostatic?},  International Journal of Solids and
  Structures \textbf{47} (2010), 2745--2754.

\bibitem{jrtk} 
D.J.~Jacobs, A.~Rader, M.~Thorpe, and L.A.~Kuhn, \emph{Protein flexibility predictions using graph theory},
Proteins: Structure, Function, and Bioinformatics, \textbf{44} (2001),150--165.

\bibitem{jj}
B.~Jackson and T.~Jord{\'a}n,
\emph{The generic rank of body-bar-and-hinge frameworks},
European J.~Combin., \textbf{31} (2010), 574--588.
% \bibitem{jkt}
%T. Jordan, V. Kaszanitzky and S. Tanigawa, \emph{Gain-sparsity and symmetry-forced rigidity in the plane}, The EGRES technical report, TR-2012-17.


%\bibitem{KG2}
%R.D. Kangwai and S.D. Guest, \emph{Symmetry-adapted equilibrium matrices}, International Journal of
%  Solids and Structures \textbf{37} (2000), 1525--1548.


\bibitem{kattan} N. Katoh and S. Tanigawa, \emph{A proof of the molecular conjecture}, Discrete $\&$ Computational Geometry \textbf{45} (2011), No. 4, 647--700.

\bibitem{Lamanbib}
G. Laman, \emph{On graphs and rigidity of plane skeletal structures}, J. Engrg.
Math. \textbf{4} (1970), 331-340


\bibitem{leestrei} A. Lee and I. Streinu, \emph{Pebble game algorihms and $(k, l)$-sparse graphs}, 
Discrete Mathematics, \textbf{308} (2008), 1425--1437.

%\bibitem{lovyem} L. Lov\'asz and Y. Yemini, \emph{On generic rigidity in the plane}, SIAM J. Alg. Disc. Methods 3 (1982), 91--98.



%\bibitem{MT3} J.~Malestein and L.~Theran, \emph{Generic rigidity of frameworks with orientation-preserving crystallographic symmetry},
%preprint, arXiv:1108.2518, 2011


% \bibitem{MT2}
%\bysame, \emph{Generic rigidity of reflection frameworks}, preprint, arXiv:1203.2276, 2012.

%\bibitem{MT1}
%\bysame, \emph{Generic rigidity with forced symmetry and sparse colored graphs}, preprint, arXiv:1203.0772, 2012.


%\bibitem{nr} T. Nixon and E. Ross, \emph{Periodic rigidity on a variable torus using inductive constructions}, preprint, arXiv:1204.1349, 2012.

\bibitem{oxley}
J.~Oxley,
\newblock \emph{Matroid theory},
\newblock Oxford University Press, USA, 2nd edition, 2011.

%\bibitem{owen}
%J.C. Owen and S.C. Power, \emph{Frameworks, symmetry and rigidity}, Int. J. Comput. Geom. Appl. \textbf{20} (2010), 723--750.

\bibitem{portaetal} J. Porta, L. Ros, B. Schulze, A. Sljoka and W. Whiteley, \emph{On the Symmetric Molecular Conjectures}, to appear in Computational Kinematics: International Workshop on Computational Kinematics (CK2013), Springer Verlag, 175-184, 2014

 

% \bibitem {ER} E. Ross, \emph{Geometric and combinatorial rigidity of periodic frameworks as graphs on the torus}, Ph. D. thesis, York University, Toronto, May 2011.

 %\bibitem{rsw} E. Ross, B. Schulze, and W. Whiteley, \emph{Finite motions from periodic frameworks with added symmetry}, Int. J. of Solids and Structures, \textbf{48} (2011), 1711 �V1729.

%\bibitem{BS2}
%B.~Schulze, \emph{Block-diagonalized rigidity matrices of symmetric frameworks and
%  applications}, Contributions to Algebra and Geometry \textbf{51} (2010), No. 2, 427--466.

%\bibitem{BS1}
%\bysame, \emph{Injective and non-injective realizations with symmetry},
%  Contributions to Discrete Mathematics \textbf{5} (2010), 59--89.

\bibitem{BS3}
B.~Schulze, \emph{Symmetric versions of {L}aman's {T}heorem}, Discrete and Computational Geometry \textbf{44} (2010), No. 4, 946--972.

\bibitem{BS4} \bysame, \emph{Symmetric Laman theorems for the groups $C_2$ and $C_s$},
The Electronic Journal of Combinatorics \textbf{17} (2010), No. 1, R154, 1--61.

%\bibitem{BS6}
%\bysame, \emph{Symmetry as a sufficient condition for a finite flex}, SIAM Journal on Discrete Mathematics \textbf{24} (2010), No. 4, 1291--1312.

\bibitem{schtan} B.~Schulze and S.~Tanigawa, \emph{Infinitesimal rigidity of symmetric bar-join frameworks}, preprint
(see Sections 1-6 in arXiv:1308.6380), 2013


\bibitem{BSWW} B.~Schulze and W.~Whiteley, \emph{The orbit rigidity matrix of a symmetric framework},
Discrete and Computational Geometry \textbf{46} (2011), No. 3, 561--598.



\bibitem{dimers} B.~Schulze, A. Sljoka and W. Whiteley, \emph{How does symmetry impact the flexibility of proteins?}, Phil. Transa. Royal Soc. A \textbf{372} (2014), No. 2008, 20120041.

%\bibitem{tan}
%S. Tanigawa, \emph{Matroids of gain graphs in applied discrete geometry}, arXiv:1207.3601, 2012.

\bibitem{tay89}
T.-S. Tay,
\newblock \emph{Linking $(n- 2)$-dimensional panels in $n$-space II:$(n- 2,2)$-frameworks and body and hinge structures},
\newblock {Graphs and Combinatorics}, 5 (1989), 245--273.

\bibitem{tay91}
\bysame,
\newblock \emph{Linking $(n- 2)$-dimensional panels in $n$-space I:$(k-1, k)$-graphs and $(k-1, k)$-frames},
\newblock {Graphs and Combinatorics}, 7 (1991), 289--304.


% \bibitem{tay} \bysame, \emph{A New Proof of Laman's Theorem}, Graphs and Combinatorics \textbf{9} (1993), 365--370.

\bibitem{Tay84}  \bysame, \emph{Rigidity of multi-graphs, linking rigid bodies in $n$-space}, J. Comb. Theory, B \textbf{36} (1984), 95--112.

 \bibitem{TW1} T.-S. Tay and W. Whiteley,  \emph{Recent advances in generic rigidity of structures}, Structural Topology \textbf{9} (1985), 31--38.



% \bibitem{LT}
%L.~Theran, \emph{Henneberg constructions and covers of cone-Laman graphs}, preprint, arXiv:1204.0503,
% 2012.


\bibitem{ww} N. White and W. Whiteley, \emph{
The algebraic geometry of motions of bar-and-body frameworks}, SIAM Journal on Algebraic Discrete Methods \textbf{8} (1987), 1--32.

\bibitem{wwmatun} W. Whiteley,  \emph{The union of matroids and the rigidity of frameworks}, SIAM Journal on Discrete
Mathematics \textbf{1} (1988), 237--255.

\bibitem{W1}
\bysame, \emph{Some {M}atroids from {D}iscrete {A}pplied {G}eometry},
  Contemporary Mathematics, AMS \textbf{197} (1996), 171--311.

\bibitem{Wbio}  \bysame,  \emph{Counting out the flexibility of proteins},
Physical Biology \textbf{2} (2005), 116--126.



%\bibitem{zaslavsky1989biased}
%T.~Zaslavsky,
%\emph{Biased graphs "{I}": Bias, balance, and gains},
% J.~Combin.~Theory Ser.~B, \textbf{47} (1989), 32--52.

%\bibitem{zaslavsky1991biased}
%\bysame,
%\emph{Biased graphs "{II}": The three matroids},
%J.~Combin.~Theory Ser.~B, \textbf{51} (1991), 46--72.
%
%\bibitem{zaslavsky2003biased}
%\bysame, \emph{Biased graphs "{IV}": Geometrical realizations},
%J.~Combin.~Theory Ser.~B, \textbf{89} (2003), 231--297.


\end{thebibliography}
\end{document}